\documentclass[10pt]{amsart}
\usepackage{amsfonts}
\usepackage{graphics,graphicx}
\usepackage{caption} 
\usepackage{epsfig}
\usepackage[utf8]{inputenc}
\usepackage[T1]{fontenc}
\usepackage{color}
\usepackage{amssymb}
\usepackage{mathrsfs}
\usepackage{amsthm}
\usepackage{amsmath}
\usepackage{amssymb}
\usepackage{latexsym}
\usepackage{bbold}
\makeatletter
\@addtoreset{equation}{section} %reinitialise le compteur equation quand section est incremente
\makeatother
\newtheorem{lem}{Lemma}[section]
\newtheorem{prop}{Proposition}[section]
\newtheorem{cor}{Corollary}[section]
\newtheorem{thm}{Theorem}[section]
\newfont{\sBlackboard}{msbm10 scaled 900}

\newcommand{\dd}     {{\rm d}}
\newcommand{\mylabel}[1]{\label{#1}
            \ifx\undefined\stillediting
            \else \fbox{$#1$}\fi }
\newcommand{\BE}{\begin{equation}}

\newcommand{\EEQ}{\end{equation}}
\newcommand{\rfb}[1]{\mbox{\rm
   (\ref{#1})}\ifx\undefined\stillediting\else:\fbox{$#1$}\fi}

\newfont{\Blackboard}{msbm10 scaled 1200}
\newcommand{\bl}[1]{\mbox{\Blackboard #1}}
\newfont{\roma}{cmr10 scaled 1200}

\def\CC{\rm \hbox{C\kern-.56em\raise.4ex
         \hbox{$\scriptscriptstyle |$}\kern+0.5 em }}
\newcommand{\ud}{\mathrm{d}}
\newcommand{\be}{\begin{equation}}
\newcommand{\ee}{\end{equation}}
\newcommand{\beq}{\begin{eqnarray}}
\newcommand{\eeq}{\end{eqnarray}}
\newcommand{\beqs}{\begin{eqnarray*}}
\newcommand{\eeqs}{\end{eqnarray*}}
\newcommand{\bt}{\begin{thm}}
\newcommand{\et}{\end{thm}}
\newcommand{\br}{\begin{remark}}
\newcommand{\er}{\end{remark}}
\newcommand{\bc}{\begin{cor}}
\newcommand{\ec}{\end{cor}}
\newcommand{\el}{\end{lem}}
\newcommand{\bd}{\begin{definition}}
\newcommand{\ed}{\end{definition}}

\newcommand{\mm}    {{\hbox{\hskip 0.5pt}}}

\newcommand{\bluff} {{\hbox{\raise 15pt \hbox{\mm}}}}
\makeatletter
\def\section{\@startsection {section}{1}{\z@}{-3.5ex plus -1ex minus
    -.2ex}{2.3ex plus .2ex}{\large\bf}}
\makeatother
\def\be{\begin{equation}}
\def\ee{\end{equation}}

\def\ds{\displaystyle}

\newcommand{\re}{\mathrm{Re}}

\newcommand{\e}{\mathrm{e}}

\newcommand{\R}{\bl{R}}
\newcommand{\N}{\bl{N}}
\newcommand{\op}{\mathrm{op}}

\begin{document}

\thispagestyle{empty}
\title[Wave equation with singular Kelvin-Voigt damping]{Stabilization for the wave equation with singular Kelvin-Voigt damping}
\author{Ka\"{\i}s AMMARI}
\address{UR Analysis and Control of Pde, UR 13ES64, Department of Mathematics, Faculty of Sciences of Monastir, University of Monastir, 5019 Monastir, Tunisia} 
\email{kais.ammari@fsm.rnu.tn} 

\author{Fathi HASSINE}
\address{UR Analysis and Control of Pde, UR 13ES64, Department of Mathematics, Faculty of Sciences of Monastir, University of Monastir, 5019 Monastir, Tunisia}
\email{fathi.hassine@fsm.rnu.tn}

\author{Luc ROBBIANO}
\address{Laboratoire de Math\'ematiques, Universit\'e de Versailles Saint-Quentin en Yvelines, 78035 Versailles, France}
\email{luc.robbiano@uvsq.fr}

%\date{}

\begin{abstract}
We consider the wave equation with Kelvin–Voigt damping in a bounded domain. The exponential stability result proposed by Liu and Rao \cite{liu-rao2} or T\'ebou \cite{tebou} for that system assumes that the damping is localized in a neighborhood of the whole or a part of the boundary under some consideration. In this paper we propose to deal with this geometrical condition by considering a singular Kelvin-Voigt damping which is localized faraway from the boundary. In this particular case it was proved by Liu and Liu \cite{liu-liu} the lack of the uniform decay of the energy. However, we show that the energy of the wave equation decreases logarithmically to zero as time goes to infinity. Our method is based on the frequency domain method. The main feature of our contribution is to write the resolvent problem as a transmission system to which we apply a specific Carleman estimate. 
\end{abstract}


\maketitle

\tableofcontents

%%%%%%%%%%%%%%%%%%%%%%%%%%%%%%%%%%%%%%%%%%%%%%%%%%%%%%%%%%%%%%%%%%%%%%%%%%%%%%%%%%%%%%%%%%%%%%%%%%%%%%%%%%%%%%%%%%%%%%%%%%%%%%%%%%%%%%%%%%%%%%%%%%%%%%%%%%%%%%%%%%%%%%%%%%%%SECTION%%%%%%%%%%%%%%%%%%%%%%%%%%%%%%%%%%%%%%%%%%%%%%%%%%%%%%%%%%%%%%%%%%%%%%%%%%%%%%%%%%%%%%%%%%%%%%%%%%%%%%%%%%%%%%%%%%%%%%%%%%%%%%%%%%%%%%%%%%%%%%%%%%%%%%%%%%%%%%%%%%%%%%%%%%%%%%%%%%%%%%%%
\section{Introduction and main results}
\setcounter{equation}{0}
There are several mathematical models representing physical damping. The most often encountered type of damping in vibration studies are linear viscous damping \cite{ammari-niciase, blr, lebeau, lebeau-robbiano2} and Kelvin-Voigt damping \cite{hassine1,liu-liu, liu-rao1,liu-rao2} which are special cases of proportional damping. Viscous damping usually models external friction forces such as air resistance acting on the vibrating structures and is thus called "external damping", while Kelvin-Voigt damping originate from the internal friction of the material of the vibrating structures and thus called "internal damping" or "material damping". This type of material is encountered in real life when one uses patches to suppress vibrations, the modeling aspect of which may be found in \cite{banks-smith-wang}. This type of question was examined in the one-dimensional setting in \cite{liu-liu} where it was shown that the longitudinal motion of an Euler-Bernoulli beam modeled by a locally damped wave equation with Kelvin-Voigt damping is not exponentially stable when the junction between the elastic part and the viscoelastic part of the beam is not smooth enough. Later on, the wave equation with Kelvin-Voigt damping in the multidimensional setting was examined in \cite{liu-rao2}; in particular, those authors showed the exponential decay of the energy by assuming that the damping region is a neighborhood of the whole boundary. Later on, it was shown that the exponential decay of the energy could be obtained with just imposing that the damping is a neighborhood of part of the boundary \cite{tebou}.

Let $\Omega \subset \R^{n}$, $n \geq 2,$ be a bounded domain with a sufficiently smooth boundary $\Gamma=\partial \Omega$. Let $\omega$ be an no empty and open subset of $\Omega$ with smooth boundary $\mathcal{I}=\partial\omega$ (see Figure \ref{fig1}). % Let $\nu$ the unit outward normal vector to $\omega$.

Consider the damping wave system
\begin{equation}
\label{wave1}
\partial_t^2 u - \Delta u - \, \mathrm{div}(a(x) \, \nabla \partial_t u) = 0, \, \Omega \times (0,+\infty), 
\end{equation}
\begin{equation}
\label{wave2}
u = 0, \, \partial \Omega \times (0,+\infty),
\end{equation}
\begin{equation}
\label{wave3}
u(x,0) = u^0, \, \partial_t u(x,0) = u^1 (x), \, \Omega,
\end{equation}
where $a(x)=d\,\mathbb{1}_{\omega}(x)$ and $d>0$ is a constant.

\medskip

%The system can be rewritten as following by denoting $u=u_{1}\,1_\omega+u_{2}\, 1_{\Omega\setminus\bar{\omega}}$,
%\be
%\label{wave1tbis}
%\partial_t^2 u_1 - \Delta u_1 - d \, \Delta \partial_t u_1 = 0, \, \omega  \times (0,+\infty), 
%\ee
%\be
%\label{wavet1}
%\partial_t^2 u_2 - \Delta u_2 = 0, \, \left[\Omega \setminus \bar{\omega} \right] \times (0,+\infty), 
%\ee
%\be
%\label{wavet2}
%u_1 = u_2, \, \partial_{\nu}u_{1}-\partial_{\nu}u_{2}=-d\,\partial_{\nu}\partial_{t}u_{1}, \,  \partial \omega \times (0,+\infty),
%\ee
%\be
%\label{wavet21}
%u = 0, \, \Gamma \times (0,+\infty),
%\ee
%\be
%\label{wavett3}
%u_1^0 (x,0) = u^0_1, \, u_1^1 (x,0) = u_1^1, \,  \Omega \setminus \bar{\omega},
%\ee
%\be
%\label{wavet3}
%u_2^0 (x,0) = u^0_1, \, u_2^1 (x,0) = u_2^1, \, \omega \times (0,+\infty). 
%\ee
 
\begin{figure}[htbp]
\includegraphics[scale=0.8]{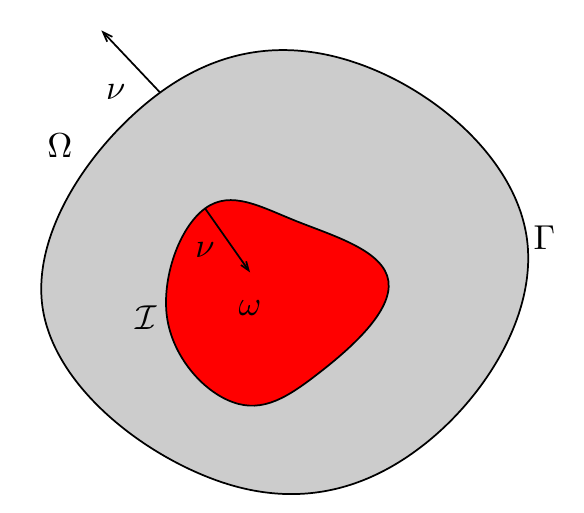} 
\caption{The domain $\Omega$}
\label{fig1}
\end{figure}

System \eqref{wave1}-\eqref{wave3}, involving a constructive viscoelastic damping $\mathrm{div}(a(x)\nabla u_{t})$, models the vibrations of an elastic body which has one part made of viscoelastic material. In the case of global viscoelastic damping $(a>0$), the wave equation \eqref{wave1}-\eqref{wave3} generates an analytic semigroup, and the spectrum of which is contained in a sector of the left half complex plan (see \cite{chen-liu-liu}). While the situation of local viscoelastic damping is more delicate due to the unboundedness of the viscoelastic damping and the discontinuity of the materials.

In \cite{liu-liu}, it was proved that the energy of an one-dimensional wave equation with local viscoelastic damping does not decay uniformly if the damping coefficient $a$ is discontinuous across the interface of the materials. Because of the discontinuity of the materials across the interface, the dissipation is badly transmitted from the viscoelastic region to the elastic region, where the energy decays slowly. Nevertheless, this does not contradict the well-known ``geometric optics'' condition in \cite{blr}, since the viscoelastic damping is unbounded in the energy space. The loss of uniform stability is caused by the discontinuity of material properties across the interface and the unboundedness of the viscoelastic damping. In this paper, we prove a logarithmically decay of energy. Our idea is transform the resolvent problem of system \eqref{wave1}-\eqref{wave2} to a transmission system to be able to quantify the discontinuity of the material properties across the interface through the so-called Carleman estimate. Noting that recently the same problem was treated in \cite{hassine1} where it was proved that the energy is polynomially decreases over the time but only on one-dimensional case (even for a transmission system).

We define the natural energy of $u$ solution of \rfb{wave1}-\rfb{wave3} at instant $t$ by
\begin{equation*}
E(u,t)=
\frac{1}{2} \left\|(u(t),\partial_t u(t))\right\|_{H^1_0 (\Omega) \times L^2(\Omega)}^2, \, \forall \, t \geq 0.
\end{equation*}
Simple formal calculations give
\begin{equation*}
E(u,0)-E(u,t)= - \, d \,\int_{0}^{t} \int_{\omega}\left|\nabla\partial_{t} u(x,s)\right|^2 \,\ud x\,\ud s,\forall t\geq 0, 
\end{equation*}
and therefore, the energy is non-increasing function of the time variable $t$.

\medskip

%We would stress here that the novelty of this article reside in the singularity of Kelvin-Voigt damping, and its localization fare way from the boundary. Recently, many paper was appeared in this direction but only for an acting  damping near the boundary \cite{liu-liu, liu-rao1, liu-rao2} where we prove a polynomial decay of energy of the solution of the wave equation, and the lack of the uniform decay. 
\begin{thm}\label{LogStab}
For any $k\in\N^*$ there exists $C>0$ such that for any initial data $(u^{0},u^{1})\in\mathcal{D}(\mathcal{A}^{k})$ the solution $u(x,t)$ of \eqref{wave1} starting from $(u^{0},u^{1})$ satisfying
$$E(u,t)\leq\frac{C}{(\ln(2+t))^{2k}}\|(u^{0},u^{1})\|_{\mathcal{D}(\mathcal{A}^{k})}^{2},\quad\forall\,t>0,$$
where $(\mathcal{A}, \mathcal{D}(\mathcal{A}))$ is defined in Section \ref{wellposed}.
\end{thm}

%\medskip

%The main result of this paper concerns the precise asymptotic behavior of the solutions of \rfb{wave1}-\rfb{wave3}.  Our technique is based on an Carleman estimate. 

\medskip

This paper is organized as follows. In Section \ref{wellposed}, we give the proper functional setting for systems \rfb{wave1}-\rfb{wave3}, and prove that this system is well-posed. In Section \ref{carleman}, we establish some Carleman estimate which is correspond to the system \rfb{wave1}-\rfb{wave3}. Finally, in Section \ref{stab}, we study the stabilization for  \rfb{wave1}-\rfb{wave3} by resolvent method and give the explicit decay rate of the energy of the solutions of \rfb{wave1}-\rfb{wave3}.
%%%%%%%%%%%%%%%%%%%%%%%%%%%%%%%%%%%%%%%%%%%%%%%%%%%%%%%%%%%%%%%%%%%%%%%%%%%%%%%%%%%%%%%%%%%%%%%%%%%%%%%%%%%%%%%%%%%%%%%%%%%%%%%%%%%%%%%%%%%%%%%%%%%%%%%%%%%%%%%%%%%%%%%%%%%%SECTION%%%%%%%%%%%%%%%%%%%%%%%%%%%%%%%%%%%%%%%%%%%%%%%%%%%%%%%%%%%%%%%%%%%%%%%%%%%%%%%%%%%%%%%%%%%%%%%%%%%%%%%%%%%%%%%%%%%%%%%%%%%%%%%%%%%%%%%%%%%%%%%%%%%%%%%%%%%%%%%%%%%%%%%%%%%%%%%%%%%%%%%%
\section{Well-posedness and strong stability}\label{wellposed}
We define the energy space by $\mathcal{H}=H_{0}^{1}(\Omega)\times L^{2}(\Omega)$ which is endowed with the usual inner product
$$
\left\langle(u_{1},v_{1});(u_{2},v_{2})\right\rangle=\int_{\Omega}\nabla u_{1}(x).\nabla \overline{u}_{2}(x)\,\dd x+\int_{\Omega}v_{1}(x)\overline{v}_{2}(x)\,\dd x.
$$
We next define the linear unbounded operator $\mathcal{A}:\mathcal{D}(\mathcal{A})\subset\mathcal{H}\longrightarrow\mathcal{H}$ by
$$
\mathcal{D}(\mathcal{A})=\{(u,v)\in\mathcal{H}: v\in H_{0}^{1}(\Omega),\; \Delta u+\mathrm{div}(a\nabla v)\in L^{2}(\Omega)\}
$$
and
$$
\mathcal{A}(u,v)^{t}=(v,\Delta u+\mathrm{div}(a\nabla v))^{t}
$$
Then, putting $v=\partial_{t} u$, we can write \eqref{wave1}-\eqref{wave3} into the following Cauchy problem
$$
\frac{d}{dt}(u(t),v(t))^{t}=\mathcal{A}(u(t),v(t))^{t},\;(u(0),v(0))=(u^{0}(x),u^{1}(x)).
$$
\begin{thm}
The operator $\mathcal{A}$ generates a $C_{0}$-semigroup of contractions on the energy space $\mathcal{H}$.
\end{thm}
\begin{proof}
Firstly, it is easy to see that for all $(u,v)\in\mathcal{D}(\mathcal{A})$, we have
$$
\mathrm{Re}\left\langle\mathcal{A}(u,v);(u,v)\right\rangle=-\int_{\Omega}a|\nabla v(x)|^{2}\,\dd x,
$$
which show that the operator $\mathcal{A}$ is dissipative.

Next, for any given $(f,g)\in\mathcal{H}$, we solve the equation $\mathcal{A}(u,v)=(f,g)$, which is recast on the following way
\begin{equation}\label{WPwave}
\left\{\begin{array}{l}
v=f,
\\
\Delta u+\mathrm{div}(a\nabla f)=g.
\end{array}\right.
\end{equation}
It is well known that by Lax-Milgram's theorem the system \eqref{WPwave} admits a unique solution $u\in H_{0}^{1}(\Omega)$. Moreover by multiplying the second line of \eqref{WPwave} by $\overline{u}$ and integrating over $\Omega$ and using Poincar\'e inequality and Cauchy-Schwarz inequality we find that there exists a constant $C>0$ such that
$$
\int_{\Omega}|\nabla u(x)|^{2}\,\dd x\leq C\left(\int_{\Omega}|\nabla f(x)|^{2}\,\dd x+\int_{\Omega}|g(x)|^{2}\,\dd x\right).
$$
It follows that for all $(u,v)\in\mathcal{D}(\mathcal{A})$ we have
$$
\|(u,v)\|_{\mathcal{H}}\leq C\|(f,g)\|_{\mathcal{H}}.
$$
This imply that $0\in\rho(\mathcal{A})$ and by contraction principle, we easily get $R(\lambda\mathrm{I}-\mathcal{A})=\mathcal{H}$ for sufficient small $\lambda>0$. The density of the domain of $\mathcal{A}$ follows from \cite[Theorem 1.4.6]{Pazy}. Then thanks to Lumer-Phillips Theorem (see \cite[Theorem 1.4.3]{Pazy}), the operator $\mathcal{A}$ generates a $C_{0}$-semigroup of contractions on the Hilbert $\mathcal{H}$. 
\end{proof}
\begin{thm}
The semigroup $e^{t\mathcal{A}}$ is strongly stable in the energy space $\mathcal{H}$, i.e,
$$
\lim_{t\to+\infty}\|e^{t\mathcal{A}}(u_{0},v_{0})^{t}\|_{\mathcal{H}}=0,\;\forall\,(u_{0},v_{0})\in\mathcal{H}.
$$ 
\end{thm}
\begin{proof}
To show that the semigroup $(e^{t\mathcal{A}})_{t\geq 0}$ is strongly stable we only have to prove that the intersection of $\sigma(\mathcal{A})$ with $i\mathbb{R}$ is an empty set. Since the resolvent of the operator $\mathcal{A}$ is not compact (see \cite{liu-liu, liu-rao2}) but $0\in\rho(\mathcal{A})$ we only need to prove that $(i\mu I-\mathcal{A})$ is a one-to-one correspondence in the energy space $\mathcal{H}$ for all $\mu\in\mathbb{R}^{*}$. The proof will be done in two steps: in the first step we will prove the injective property of $(i\mu I-\mathcal{A})$ and in the second step we will prove the surjective property of the same operator.

i) Let $(u,v)\in\mathcal{D}(\mathcal{A})$ such that 
\begin{equation}\label{Iwave}
\mathcal{A}(u,v)^{t}=i\mu(u,v)^{t}.
\end{equation}
Then taking the real part of the scalar product of \eqref{Iwave} with $(u,v)$ we get
$$
\mathrm{Re}(i\mu\|(u,v)\|_{\mathcal{H}}^{2})=\mathrm{Re}\left\langle\mathcal{A}(u,v),(u,v)\right\rangle=-d\int_{\omega}|\nabla v|^{2}\dd x=0.
$$
which implies that
\begin{equation}\label{Dwave}
\nabla v=0 \qquad \text{ in }\,\omega.
\end{equation}
Inserting \eqref{Dwave} into \eqref{Iwave}, we obtain
\begin{equation}\label{waveI1}
\left\{\begin{array}{ll}
\mu^{2}u+\Delta u=0&\text{in }\Omega\backslash\omega,
\\
\nabla u=0&\text{in }\omega
\\
u=0&\text{on }\Gamma,
\end{array}\right.
\end{equation}

We denote by $w_{j}=\partial_{x_{j}}u$ and we derive the first and the second equations of \eqref{waveI1}, one gets
\begin{equation*}
\left\{\begin{array}{ll}
\mu^{2}w_{j}+\Delta w_{j}=0&\text{in }\Omega,
\\
w_{j}=0&\text{in }\omega.
\end{array}\right.
\end{equation*}
%\textbf{L'\' equation ci-dessus est fausse. Si on prolonge $u$ par 0, quand on calcule $\Delta u$ on va avoir un terme $(\partial_\nu u)_{|\partial \Omega}\otimes \delta_{|\partial \Omega}$. Voir la premi\`ere correction que j'ai envoy\'ee.}
Hence, from the unique continuation theorem we deduce that $w_{j}=0$ in $\Omega$ and therefore $u$ is constant in $\Omega$ and since $u_{|\Gamma}=0$ we follow that $u\equiv 0$. We have thus proved that $\mathrm{Ker}(i\mu I-\mathcal{A})=0$.
	
ii) Now given $(f,g)\in\mathcal{H}$, we solve the equation 
$$
(i\mu I-\mathcal{A})(u,v)=(f,g)
$$
Or equivalently,
\begin{equation}\label{Swave}
\left\{\begin{array}{l}
v=i\mu u-f
\\
\mu^{2}u+\Delta u+i\mu\,\mathrm{div}(a\nabla u)=\mathrm{div}(a\nabla f)-i\mu f-g.
\end{array}\right.
\end{equation}
Let's define the operator
$$
Au=-(\Delta u+i\mu\,\mathrm{div}(a\nabla u)),\quad \forall\, u\in H_{0}^{1}(\Omega).
$$
It is easy to show that $A$ is an isomorphism from $H_{0}^{1}(\Omega)$ onto $H^{-1}(\Omega)$. Then the second line of \eqref{Swave} can be written as follow
\begin{equation}\label{Eqwave}
u-\mu^{2}A^{-1}u=A^{-1}\left[g+i\mu f-\mathrm{div}(a\nabla f)\right].
\end{equation}
If $u\in\mathrm{Ker}(I-\mu^{2}A^{-1})$, then $\mu^{2}u-Au=0$. It follows that
\begin{equation}\label{Awave}
\mu^{2}u+\Delta u+i\mu\mathrm{div}(a\nabla u)=0.
\end{equation}
Multiplying \eqref{Awave} by $\overline{u}$ and integrating over $\Omega$, then by Green's formula we obtain 
$$
\mu^{2}\int_{\Omega}|u(x)|^{2}\,\dd x-\int_{\Omega}|\nabla u(x)|^{2}\,\dd x-id\mu\int_{\omega}|\nabla u(x)|^{2}\,\dd x=0.
$$
This shows that 
$$
d\int_{\omega}|\nabla u(x)|^{2}\,\dd x=0,
$$
which imply that $\nabla u=0$ in $\omega$.
\\
Inserting this last equation into~\eqref{Awave} we get
$$
\mu^{2}u+\Delta u=0,\qquad \text{in }\Omega.
$$
Once again, using the unique continuation theorem as in the first step where we recall that $u_{|\Gamma}=0$, we get $u=0$ in $\Omega$. This imply that $\mathrm{Ker}(I-\mu^{2}A^{-1})=\{0\}$. On the other hand thanks to the compact embeddings $H_{0}^{1}(\Omega)\hookrightarrow L^{2}(\Omega)$ and $L^{2}(\Omega)\hookrightarrow H^{-1}(\Omega)$ we see that $A^{-1}$ is a compact operator in $H_{0}^{1}(\Omega)$. Now thanks to Fredholm's alternative, the operator $(I-\mu^{2}A^{-1})$ is bijective in $H_{0}^{1}(\Omega)$, hence the equation \eqref{Eqwave} have a unique solution in $H_{0}^{1}(\Omega)$, which yields that the operator $(i\mu I-\mathcal{A})$ is surjective in the energy space $\mathcal{H}$. The proof is thus complete.
\end{proof}
%%%%%%%%%%%%%%%%%%%%%%%%%%%%%%%%%%%%%%%%%%%%%%%%%%%%%%%%%%%%%%%%%%%%%%%%%%%%%%%%%%%%%%%%%%%%%%%%%%%%%%%%%%%%%%%%%%%%%%%%%%%%%%%%%%%%%%%%%%%%%%%%%%%%%%%%%%%%%%%%%%%%%%%%%%%%SECTION%%%%%%%%%%%%%%%%%%%%%%%%%%%%%%%%%%%%%%%%%%%%%%%%%%%%%%%%%%%%%%%%%%%%%%%%%%%%%%%%%%%%%%%%%%%%%%%%%%%%%%%%%%%%%%%%%%%%%%%%%%%%%%%%%%%%%%%%%%%%%%%%%%%%%%%%%%%%%%%%%%%%%%%%%%%%%%%%%%%%%%%%
\section{Carleman estimate}\label{carleman}
For any $s\in\R$ we define the Sobolev space with a parameter $\tau$, $H_{\tau}^{s}$ by
$$
u(x,\tau)\in H^{s}_{\tau}\,\Longleftrightarrow\,\left\langle\xi,\tau\right\rangle^{s}\hat{u}(\xi,\tau)\in L^{2};\qquad\langle\xi,\tau\rangle^{2}=|\xi|^{2}+\tau^{2},
$$
where $\hat{u}$ denote the partial Fourier transform with respect to $x$.
The class of symbols of order $m$ defined by
$$
\mathcal{S}_{\tau}^{m}=\left\{a(x,\xi,\tau)\in\mathcal{C}^{\infty};
\;|\partial_{x}^{\alpha}\partial_{\xi}^{\beta}a(x,\xi,\tau)|\leq C_{\alpha,\beta}\langle\xi,\tau\rangle ^{m-|\beta|} \right\}
$$
and the class of tangential symbols of order $m$ by
$$
\mathcal{TS}_{\tau}^{m}=\left\{a(x,\xi',\tau)\in\mathcal{C}^{\infty}
;\;|\partial_{x}^{\alpha}\partial_{\xi'}^{\beta}a(x,\xi',\tau)|\leq C_{\alpha,\beta}\langle\xi',\tau\rangle^{m-|\beta|}   \right\}
$$
We denote by $\mathcal{O}^{m}$ (resp. $\mathcal{TO}^{m}$) the set of pseudo-differential operators $A=\op(a)$, $a\in\mathcal{S}^{m}$ (resp. $a\in\mathcal{TS}^{m}$). We shall use the symbol $\Lambda=\langle\xi',\tau\rangle=(|\xi'|^{2}+\tau^{2})^{\frac{1}{2}}$. 

Consider a bounded smooth open set $\mathcal{U}$ of $\R^{n}$ with boundary $\partial\mathcal{U}=\gamma$. We set $\mathcal{U}_{1}$ and $\mathcal{U}_{2}$ two smooth open subsets of $\mathcal{U}$ with boundaries $\partial\mathcal{U}_{1}=\gamma_{0}$ and $\partial\mathcal{U}_{2}=\gamma\cup\gamma$ such that $\overline{\gamma}_{0}\cup\overline{\gamma}=\emptyset$. We denote by $\nu(x)$ the unit outer normal to $\mathcal{U}_{2}$ if $x\in\gamma_{0}\cup\gamma$.\label{carl53}

For $\tau$ a large parameter and $\varphi_{1}$ and $\varphi_{2}$ two weight functions of class $\mathcal{C}^{\infty}$ in 
$\overline{\mathcal{U}}_{1}$ and $\overline{\mathcal{U}}_{2}$ respectively such that 
$\varphi_{1|\gamma_{0}}=\varphi_{2|\gamma_{0}}$ we denote by $\varphi(x)=\mathrm{diag}(\varphi_{1}(x),\varphi_{2}(x))$ 
and   let $\alpha$ be a non null complex number.
 We set the differential operator
$$
P=\mathrm{diag}(P_{1},P_{2})=\mathrm{diag}\left(-\Delta+\frac{\tau^{2}}{1+\alpha\tau},-\Delta-\tau^{2}\right),
$$
and its conjugate operator
$$
P(x,D,\tau)=\e^{\tau\varphi}P\e^{-\tau\varphi}=\mathrm{diag}(P_{1}(x,D,\tau),P_{2}(x,D,\tau)),
$$
with principal symbol $p(x,\xi,\tau)$ given by
\begin{align*}
p(x,\xi,\tau)&=\mathrm{diag}(p_{1}(x,\xi,\tau),p_{2}(x,\xi,\tau))
\\
&=\mathrm{diag}(|\xi|^{2}
+2i\tau \xi\nabla\varphi_{1}-\tau^{2}|\nabla\varphi_{1}|^{2},|\xi|^{2}+2i\tau\xi\nabla\varphi_{2}-\tau^{2}|\nabla\varphi_{2}|^{2}-\tau^{2}).
\end{align*}
In a small neighborhood $W$ of a point $x_{0}$ of $\gamma_{0}$, we place ourselves in normal geodesic 
coordinates and we denote by $x_{n}$ the variable that is normal to the interface $\gamma_{0}$ and by $x'$ 
the reminding spacial variables, i.e., $x=(x',x_{n})$. The interface $\gamma_{0}$ is now given by 
$\gamma_{0}=\{x\,;\;x_{n}=0\}$ where in particular we can assume that $x_{0}=(0,0)$ and that $W$ is symmetric with 
respect to $x_{n}\longmapsto-x_{n}$. We denote by
$$
W_{1}=\{x\in\R^{n},\;x_{n}>0\}\cap W,\qquad \text{and}\qquad W_{2}=\{x\in\R^{n},\;x_{n}<0\}\cap W.
$$

Next we will proceed like Bellassoued in \cite{bellassoued} and we will reduce the problem of the transmission in only one side. The operator $P_{1}$ on $W_{1}$ is written in the form
$$
P_{1}(x,D)=D_{x_{n}}^{2}+R(+x_{n},x',D_{x'})+\frac{\tau^{2}}{1+\alpha\tau}.
$$
and the operator $P_{2}$ on $W_{2}$ can be identified locally to an operator in $W_{1}$ given by
$$
P_{2}(x,D)=D_{x_{n}}^{2}+R(-x_{n},x',D_{x'})-\tau^{2}
$$
We denote the tangential operator, with the $\mathcal{C}^{\infty}$ coefficients defined in $W_{1}$ by
$$
R(x,D_{x'})=\mathrm{diag}(R(+x_{n},x',D_{x'}),R(-x_{n},x',D_{x'}))=\mathrm{diag}(R_{1}(x,D_{x'}),R_{2}(x,D_{x'}))
$$
with principal symbol $r(x,\xi')=\mathrm{diag}(r_{1}(x,\xi'),r_{2}(x,\xi'))$ where the quadratic form $r_{k}(x,\xi')$, $k=1,2$ satisfies
$$
\exists\,C>0,\quad\forall\,(x,\xi')\in W_{1}\times\R^{n-1},\quad r_{k}(x,\xi')\geq C|\xi'|^{2},\qquad k=1,2.
$$
We assume that $\varphi$ satisfies
\begin{eqnarray}
|\nabla\varphi_{k}(x)|> 0,\;\forall\,x\in\overline{W}_{1},\quad k=1,2,\label{carl2}
\\
\partial_{x_{n}}\varphi_{1}(x',0)<0\quad\text{and}\quad\partial_{x_{n}}\varphi_{2}(x',0)>0
\\
\left(\partial_{x_{n}}\varphi_{1}(x',0)\right)^{2}-\left(\partial_{x_{n}}\varphi_{2}(x',0)\right)^{2}>1,
\end{eqnarray}
The principal symbol $p(x,\xi,\tau)$ of $P(x,D,\tau)$ is now given by
$$
p(x,\xi,\tau)=\mathrm{diag}(p_{1}(x,\xi,\tau),p_{2}(x,\xi,\tau))=\left(\xi+i\tau(\partial_{x_{n}}\varphi)\right)^{2}
+r(x,\xi'+i\tau(\partial_{x'}\varphi))-\mathrm{diag}(0,\tau^{2})\,\in\mathcal{S}_{\tau}^{2},
$$
where we assume that it satisfies to following the sub-ellipticity condition
\begin{equation}\label{carl3}
\exists\,c>0,\;\forall\,(x,\xi)\in\overline{W}_{1}\times\R^{n},\;p_{k}(x,\xi,\tau)=0\,\Longrightarrow\,\left\{\mathrm{Re}(p_{k}),\mathrm{Im}(p_{k})\right\}(x,\xi,\tau)\geq c\langle \xi,\tau\rangle^3.
\end{equation}
We defined on the boundary $\{x_{n}=0\}\cap W$ the operators
\begin{equation*}
\left\{\begin{array}{ll}
\op(b_{1})w=w_{1}-w_{2}&\text{on }\{x_{n}=0\}\cap W
\\
\op(b_{2})w=\left(D_{x_{n}}+i\tau\partial_{x_{n}}\varphi_{1}\right)w_{1}+\left(D_{x_{n}}+i\tau\partial_{x_{n}}\varphi_{2}\right)w_{2}&\text{on }\{x_{n}=0\}\cap W.
\end{array}\right.
\end{equation*}
We denote by $\|v\|=\|v\|_{L^{2}(W_{2})}$ with the correspondent scalar product denoted by $(v_{1},v_{2})$. For $s\in\R$ we denote by $\|v\|_{s}^{2}=\|\op(\Lambda^{s})v\|^{2}$ and $|v|_{s}^{2}=\|v_{|x_{n}=0}\|_{s}^{2}$ such that when $s=0$ the norm $|v|_{0}$ with the scalar product $(v_{1},v_{2})_{0}=(v_{1|x_{n}=0},v_{2|x_{n}=0})$ will be denoted simply $|v|$. Finally, we denote by $|v|_{1,0,\tau}^{2}=|v|_{1}^{2}+|D_{n}v|^{2}$.

Before proving the Carleman estimate we recall the following theorem given by \cite[Theorem 2.3]{rousseau-robbiano}.
\begin{prop}
Let $\varphi$ satisfies \eqref{carl2}-\eqref{carl3}. Then there exist $C>0$ and $\tau_{0}>0$ such that for any $\tau\geq\tau_{0}$ we have the following estimate
\begin{equation}\label{carl6}
\tau^{3}\|w\|^{2}+\tau\|\nabla w\|^{2}\leq C\left(\|P(x,D,\tau)w\|^{2}+\tau|w|_{1,0,\tau}^{2}\right)
\end{equation}
and
\begin{equation}\label{carl7}
\tau^{3}\|w\|^{2}+\tau\|\nabla w\|^{2}+\tau|w|_{1,0,\tau}^{2}\leq C\left(\|P(x,D,\tau)w\|^{2}+\tau|\op(b_{1})w|_{1}^{2}+\tau|\op(b_{2})w|^{2}\right)
\end{equation}
for any $w\in\mathcal{C}_{0}^{\infty}(K)$ where $K\subset\overline{W}_{1}$ is a compact subset.
\end{prop}
Now we are ready to state our local Carleman estimate whose main ingredients are estimates \eqref{carl6} and \eqref{carl7}. In fact, the Carleman estimate established here is an estimate analogous the previous one but with another scale of Sobolev spaces.
\begin{thm}
Let $\varphi$ satisfies \eqref{carl2}-\eqref{carl3}. There exist $C>0$ and $\tau_{0}>0$ such that for any $\tau\geq\tau_{0}$ we have the following estimate
\begin{equation}\label{carl51}
\tau^{3}\|w\|^{2}+\tau\|\nabla w\|^{2}+\tau^{2}|w|_{\frac{1}{2}}^{2}+\tau^{2}|D_{x_{n}}w|_{-\frac{1}{2}}^{2}\leq C\left(\|P(x,D,\tau)w\|^{2}
+\tau^2 |\op(b_{1})w|_{\frac{1}{2}}^{2}+\tau|\op(b_{2})w|^{2}\right)
\end{equation}
for any $w\in\mathcal{C}_{0}^{\infty}(K)$ where $K\subset\overline{W}_{1}$ is a compact subset.
\end{thm}
\begin{proof}
We can write the operator $P(x,D,\tau)$ as follow
$$
P(x,D,\tau)=D_{x_{n}}^{2}+R+\tau c_{0}(x)D_{x_{n}}+\tau C_{1}(x)+\tau^{2} c_{0}'(x),
$$
where $c_{0},\,c_{0}'\in \mathcal{TO}^{0}$, $C_{1}\in\mathcal{TO}^{1}$ and $R\in \mathcal{TO}^{2}$ with $\ds R=\sum_{j,k=1}^{n-1}a_{j,k}D_{x_{j}}D_{x_{k}}$. Let $v\in\mathcal{C}_{0}^{\infty}({W}_{1})$, then we have
\begin{equation}\label{carl8}
\begin{split}
\|(D_{x_{n}}^{2}+R)\op(\Lambda^{-\frac{1}{2}})v\|^{2}&\leq C\Big(\|P\op(\Lambda^{-\frac{1}{2}})v\|^{2}+\tau^{2}\|\op(\Lambda^{\frac{1}{2}})v\|^{2}
\\
&\quad +\tau^{2}\|D_{x_{n}}\op(\Lambda^{-\frac{1}{2}})v\|^{2}+\tau^{4}\|\op(\Lambda^{-\frac{1}{2}})v\|^{2}\Big).
\end{split}
\end{equation}
%\textbf{Il faut introduire une troncature et remplacer $\op(\Lambda^{-\frac{1}{2}})v$  par $\chi \op(\Lambda^{-\frac{1}{2}})v$  car l'in\'egalit\'e qu'on utilise n'est vraie que pour les fonctions \`a support compact. Je n'ai pas modifi\'e dans ce qui suit. Je pense que \c ca ne donne que quelques termes d'erreurs suppl\'ementaires.}

We can estimate the three last terms of the right hand side of \eqref{carl8} as follow
$$
\tau^{2}\|D_{x_{n}}\op(\Lambda^{-\frac{1}{2}})v\|^{2}+\tau^{4}\|\op(\Lambda^{-\frac{1}{2}})v\|^{2}\leq C(\tau\|D_{x_{n}}v\|^{2}+\tau^{3}\|v\|^{2}),
$$
and
\begin{equation}\label{carl28}
\tau^{2}\|\op(\Lambda^{\frac{1}{2}})v\|^{2}=\tau^{2}\left(\frac{1}{\sqrt{\tau}}\op(\Lambda)v,\sqrt{\tau}v\right)\leq C\left(\tau\|\op(\Lambda)v\|^{2}+\tau^{3}\|v\|^{2}\right)\leq C\tau\|\op(\Lambda)v\|^{2}.
\end{equation}
Then following  \eqref{carl8} we obtain
\begin{equation}\label{carl9}
\|(D_{x_{n}}^{2}+R)\op(\Lambda^{-\frac{1}{2}})v\|^{2}\leq C\left(
\|P\op(\Lambda^{-\frac{1}{2}})v\|^{2}+\tau\|\op(\Lambda)v\|^{2}+\tau^{3}\|v\|^{2}
+\tau\| D_{x_{n}}   v\|^{2}
\right).
\end{equation}
Combining \eqref{carl6} and \eqref{carl9} and using the fact that $\tau(\|\op(\Lambda)v\|^{2}+\|D_{x_{n}}v\|^{2})\sim\tau^{3}\|v\|^{2}+\tau\|\nabla v\|^{2}$ we obtain
\begin{equation}\label{carl10}
\|(D_{x_{n}}^{2}+R)\op(\Lambda^{-\frac{1}{2}})v\|^{2}\leq C\left(\|P\op(\Lambda^{-\frac{1}{2}})v\|^{2}+\|Pv\|^{2}+\tau|v|_{1,0,\tau}^{2}\right).
\end{equation}
We can write
\begin{equation}\label{carl11}
\begin{split}
P\op(\Lambda^{-\frac{1}{2}})v&=\op(\Lambda^{-\frac{1}{2}})Pv+[P,\op(\Lambda^{-\frac{1}{2}})]v=\op(\Lambda^{-\frac{1}{2}})Pv+[R,\op(\Lambda^{-\frac{1}{2}})]v
\\
&+\tau[c_{0}(x)D_{x_{n}},\op(\Lambda^{-\frac{1}{2}})]v+\tau[C_{1}(x),\op(\Lambda^{-\frac{1}{2}})]v+\tau^{2}[c_{0}'(x),\op(\Lambda^{-\frac{1}{2}})]v.
\end{split}
\end{equation}
Since $[R,\op(\Lambda^{-\frac{1}{2}})]\in\mathcal{TO}^{\frac{1}{2}}$, then following to \eqref{carl6} we have
\begin{equation}\label{carl12}
\left\|[R,\op(\Lambda^{-\frac{1}{2}})]v\right\|^{2}\leq C\|\op(\Lambda^{\frac{1}{2}})v\|^{2}\leq C\|\op(\Lambda)v\|^{2}\leq C\left(\|Pv\|^{2}+\tau|v|_{1,0,\tau}^{2}\right).
\end{equation}
Since $[c_{0}(x)D_{x_{n}},\op(\Lambda^{-\frac{1}{2}})]\in\mathcal{TO}^{-\frac{1}{2}}D_{x_{n}}$, then following to \eqref{carl6} we have
\begin{equation}\label{carl13}
\tau^{2}\left\|[c_{0}(x)D_{x_{n}},\op(\Lambda^{-\frac{1}{2}})]v\right\|^{2}\leq C\tau^{2}\|\op(\Lambda^{-\frac{1}{2}})D_{x_{n}}v\|^{2}\leq C\tau\|D_{x_{n}}v\|^{2}\leq C\left(\|Pv\|^{2}+\tau|v|_{1,0,\tau}^{2}\right).
\end{equation}
Since $[C_{1}(x),\op(\Lambda^{-\frac{1}{2}})]\in\mathcal{TO}^{-\frac{1}{2}}$ then following to \eqref{carl6} we have
\begin{equation}\label{carl14}
\tau^{2}\left\|[C_{1}(x),\op(\Lambda^{-\frac{1}{2}})]v\right\|^{2}\leq C\tau^{2}\|\op(\Lambda^{-\frac{1}{2}})v\|^{2}\leq C\tau\|v\|^{2}\leq C\left(\|Pv\|^{2}+\tau|v|_{1,0,\tau}^{2}\right).
\end{equation}
Since $[c_{0}'(x),\op(\Lambda^{-\frac{1}{2}})]\in\mathcal{TO}^{-\frac{3}{2}}$, then following to \eqref{carl6} we have
\begin{equation}\label{carl15}
\tau^{4}\left\|[c_{0}'(x),\op(\Lambda^{-\frac{1}{2}})]v\right\|\leq C\tau^{4}\|\op(\Lambda^{-\frac{3}{2}})v\|^{2}\leq C\tau^{3}\|v\|^{2}\leq C\left(\|Pv\|^{2}+\tau|v|_{1,0,\tau}^{2}\right).
\end{equation}
From \eqref{carl11}-\eqref{carl15}, one gets
\begin{equation}\label{carl16}
\|P\op(\Lambda^{-\frac{1}{2}})v\|^{2}\leq C\left(\|Pv\|^{2}+\tau|v|_{1,0,\tau}^{2}\right).
\end{equation}
Then the combination of \eqref{carl10} and \eqref{carl16} gives
\begin{equation}\label{carl17}
\|(D_{x_{n}}^{2}+R)\op(\Lambda^{-\frac{1}{2}})v\|^{2}\leq C\left(\|Pv\|^{2}+\tau|v|_{1,0,\tau}^{2}\right).
\end{equation}
In another hand, by integration by parts we find
\begin{align}\label{carl18}
\|(D_{x_{n}}^{2}+R)\op(\Lambda^{-\frac{1}{2}})v\|^{2}
&=\|D_{x_{n}}^{2}\op(\Lambda^{-\frac{1}{2}})v\|^{2}  \notag \\
&\quad +\|R\op(\Lambda^{-\frac{1}{2}})v\|^{2}
\notag  % \\
%& \quad 
+2\re(D_{x_{n}}^{2}\op(\Lambda^{-\frac{1}{2}})v,R\op(\Lambda^{-\frac{1}{2}})v)  \notag \\
& =\|D_{x_{n}}^{2}\op(\Lambda^{-\frac{1}{2}})v\|^{2}+\|R\op(\Lambda^{-\frac{1}{2}})v\|^{2}
\\
&\quad +2\re\Big(i\big(D_{x_{n}}v,R\op(\Lambda^{-1})v\big)_{0}
+i\big(D_{x_{n}}v,[\op(\Lambda^{-\frac{1}{2}}),R]\op(\Lambda^{-\frac{1}{2}})v\big)_{0}\Big)
  \notag \\
&\quad +2\re\big(RD_{x_{n}}\op(\Lambda^{-\frac{1}{2}})v,D_{x_{n}}\op(\Lambda^{-\frac{1}{2}})v\big)  \notag
\\
&\quad +2\re \big(D_{x_{n}}\op(\Lambda^{-\frac{1}{2}})v,[D_{x_{n}},R]\op(\Lambda^{-\frac{1}{2}})v\big).    \notag
\end{align}
Let $\chi_{0}\in\mathcal{C}_{0}^{\infty}(\overline{\R_{+}^{n}})$ be a positive function such that $\chi_{0}\equiv 1$ in the support of $v$ then by integration by parts and using the fact that $(1-\chi_{0})v\equiv 0$ we obtain
\begin{align}\label{carl19}
%\begin{split}
\left\|\op(\Lambda^{\frac{3}{2}})v\right\|^{2}& =(\op(\Lambda^{2})\op(\Lambda^{\frac{1}{2}})v,\op(\Lambda^{\frac{1}{2}})v)
=\tau^{2}\left\|\op(\Lambda^{\frac{1}{2}})v\right\|^{2}+\sum_{j=1}^{n-1}\left(D_{x_{j}}^{2}\op(\Lambda^{\frac{1}{2}})v,\op(\Lambda^{\frac{1}{2}})v\right)
\notag \\
&=\tau^{2}\left\|\op(\Lambda^{\frac{1}{2}})v\right\|^{2}
+\sum_{j=1}^{n-1}\left(D_{x_{j}}\op(\Lambda^{\frac{1}{2}})v,D_{x_{j}}\op(\Lambda^{\frac{1}{2}})v\right)
\notag \\
&=\tau^{2}\left\|\op(\Lambda^{\frac{1}{2}})v\right\|^{2}
+\sum_{j=1}^{n-1}\left(\chi_{0} D_{x_{j}}\op(\Lambda^{\frac{1}{2}})v,D_{x_{j}}\op(\Lambda^{\frac{1}{2}})v\right)
\\
&\quad +\sum_{j=1}^{n-1}\left([(1-\chi_{0}),D_{x_{j}}\op(\Lambda^{\frac{1}{2}})]v,D_{x_{j}}\op(\Lambda^{\frac{1}{2}})v\right)  \notag
%\end{split}
\end{align}
Since $[(1-\chi_{0}),D_{x_{j}}\op(\Lambda^{\frac{1}{2}})]\in\mathcal{TO}^{\frac{1}{2}}$ and $D_{x_{j}}\op(\Lambda^{\frac{1}{2}})\in\mathcal{TO}^{\frac{3}{2}}$ for $j=1,\ldots,n-1$, we show
\begin{equation}\label{carl20}
\left|\sum_{j=1}^{n-1}\left([(1-\chi_{0}),D_{x_{j}}\op(\Lambda^{\frac{1}{2}})]v,D_{x_{j}}\op(\Lambda^{\frac{1}{2}})v\right)\right|\leq C\|\op(\Lambda)v\|^{2}.
\end{equation}
We recall that $\ds\sum_{j,k=1}^{n-1}\chi_{0} a_{j,k}D_{x_{j}}v\overline{D_{x_{k}}v}\geq c\chi_{0}\sum_{j=1}^{n-1}|D_{x_{j}}v|^{2}$, for some constant $c>0$ and using the fact that $[\chi_{0},a_{j,k}D_{x_{j}}\op(\Lambda^{\frac{1}{2}})]\in\mathcal{TO}^{\frac{1}{2}}$ and $D_{x_{k}}\op(\Lambda^{\frac{1}{2}})\in\mathcal{TO}^{\frac{3}{2}}$, we obtain
%\textbf{Je pense qu'il faut remplacer $\chi$ par $\chi_0$ ci-dessous, et $w$ par $v$.}
\begin{align}\label{carl21}
\sum_{j=1}^{n-1}\left(\chi_{0} D_{x_{j}}\op(\Lambda^{\frac{1}{2}})v,D_{x_{j}}\op(\Lambda^{\frac{1}{2}})v\right)&\leq C\sum_{j,k=1}^{n-1}\left(\chi_{0} a_{j,k}D_{x_{j}}\op(\Lambda^{\frac{1}{2}})v,D_{k}\op(\Lambda^{\frac{1}{2}})v\right)   
\\
&\leq C\sum_{j,k=1}^{n-1}\left([\chi_{0},a_{j,k}D_{x_{j}}\op(\Lambda^{\frac{1}{2}})]v,D_{x_{k}}\op(\Lambda^{\frac{1}{2}})v\right)\notag
\\
&+C\sum_{j,k=1}^{n-1}\left(a_{j,k}D_{x_{j}}\op(\Lambda^{\frac{1}{2}})v,D_{x_{k}}\op(\Lambda^{\frac{1}{2}})v\right)\notag
\\
&\leq C\sum_{j,k=1}^{n-1}\left(a_{j,k}D_{x_{j}}\op(\Lambda^{\frac{1}{2}})v,D_{x_{k}}\op(\Lambda^{\frac{1}{2}})v\right)+C\|\op(\Lambda)v\|^{2}.\notag
\end{align}
Integrating by parts the first term of the right hand side of \eqref{carl21}, with
$ \ds R=\sum_{j,k=1}^{n-1}a_{j,k}D_{x_{j}}D_{x_{k}}$, one gets
\begin{align}\label{carl22}
%\begin{split}
\sum_{j,k=1}^{n-1}\left(a_{j,k}D_{x_{j}}\op(\Lambda^{\frac{1}{2}})v,D_{x_{k}}\op(\Lambda^{\frac{1}{2}})v\right)
&=(R\op(\Lambda^{\frac{1}{2}})v,\op(\Lambda^{\frac{1}{2}})v)
\\
&\quad +\sum_{j,k=1}^{n-1}\left([D_{x_{k}},a_{j,k}]D_{x_{j}}\op(\Lambda^{\frac{1}{2}})v,\op(\Lambda^{\frac{1}{2}})v\right).  \notag
%\end{split}
\end{align}
Since $[D_{x_{k}},a_{j,k}]D_{x_{j}}\op(\Lambda^{\frac{1}{2}})\in\mathcal{TO}^{\frac{3}{2}}$, then
\begin{equation}\label{carl23}
\left|\sum_{j,k=1}^{n-1}\left([D_{x_{k}},a_{j,k}]D_{x_{j}}\op(\Lambda^{\frac{1}{2}})v,\op(\Lambda^{\frac{1}{2}})v\right)\right|\leq C\|\op(\Lambda)v\|^{2}.
\end{equation}
Since
\begin{equation}\label{carl24}
(R\op(\Lambda^{\frac{1}{2}})v,\op(\Lambda^{\frac{1}{2}})v)=(R\op(\Lambda^{-\frac{1}{2}})v,\op(\Lambda^{\frac{3}{2}})v)+([\op(\Lambda),R]\op(\Lambda^{-\frac{1}{2}})v,\op(\Lambda^{\frac{1}{2}})v),
\end{equation}
and using the fact that $[\op(\Lambda ),R]\op(\Lambda^{-\frac{1}{2}})\in\mathcal{TO}^{\frac{3}{2}}$ and the Cauchy-Schwarz inequality, we obtain 
\begin{equation}\label{carl25}
\left|(R\op(\Lambda^{\frac{1}{2}})v,\op(\Lambda^{\frac{1}{2}})v)\right|\leq C\left(\epsilon\|\op(\Lambda^{\frac{3}{2}})v\|^{2}+\frac{1}{\epsilon}\|R\op(\Lambda^{-\frac{1}{2}})v\|^{2}+\|\op(\Lambda)v\|^{2}\right).
\end{equation}
Combining \eqref{carl19}--\eqref{carl25}, we obtain for $\epsilon$ small enough
\begin{equation}\label{carl26}
 \|R\op(\Lambda^{-\frac{1}{2}})v\|^{2}\geq C\left(\|\op(\Lambda^{\frac{3}{2}})v\|^{2}-\tau\|\op(\Lambda)v\|^{2}\right),
\end{equation}
where we have used again \eqref{carl28}. The same computation shows
\begin{equation}\label{carl27}
\re\left(RD_{x_{n}}\op(\Lambda^{-\frac{1}{2}})v,D_{x_{n}}\op(\Lambda^{-\frac{1}{2}})v\right)\geq C\left(\|D_{x_{n}}\op(\Lambda^{\frac{1}{2}})v\|^{2}-\tau\|D_{x_{n}}v\|^{2}\right).
\end{equation}
Since $[\op(\Lambda^{-\frac{1}{2}}),R]\op(\Lambda^{-\frac{1}{2}})\in\mathcal{TO}^{0}$ and $R\op(\Lambda^{-1})\in\mathcal{TO}^{1}$, we have
\begin{equation}\label{carl29}
\left|\left(D_{x_{n}}v,R\op(\Lambda^{-1})v\right)_{0}\right|+\left|\left(D_{x_{n}}v,[\op(\Lambda^{-\frac{1}{2}}),R]\op(\Lambda^{-\frac{1}{2}})v\right)_{0}\right|\leq C\left(|D_{x_{n}}v|^{2}+|v|_1^{2}\right)\leq C|v|_{1,0,\tau}^{2},
\end{equation}
and
\begin{equation}\label{carl30}
\left|\left(D_{x_{n}}\op(\Lambda^{-\frac{1}{2}})v,[D_{x_{n}},R]\op(\Lambda^{-\frac{1}{2}})v\right)\right|\leq C\|v\|^{2}+\|\nabla v\|^{2}.
\end{equation}
Putting \eqref{carl17} and \eqref{carl26}--\eqref{carl30} into \eqref{carl18}, we find
\begin{multline}\label{carl31}
%\begin{split}
\|D_{x_{n}}^{2}\op(\Lambda^{-\frac{1}{2}})v\|^{2}+\|D_{x_{n}}\op(\Lambda^{\frac{1}{2}})v\|^{2}+\|\op(\Lambda^{\frac{3}{2}})v\|^{2}
\\
\leq C\left(\|Pv\|^{2}+\tau^{3}\|v\|^{2}+\tau\|\nabla v\|^{2}+\tau|v|_{1,0,\tau}^{2}\right).
%\end{split}
\end{multline}
Following \eqref{carl7} and \eqref{carl31} will be reduced to the following estimate
\begin{multline}\label{carl32}
%\begin{split}
\|D_{x_{n}}^{2}\op(\Lambda^{-\frac{1}{2}})v\|^{2}+\|D_{x_{n}}\op(\Lambda^{\frac{1}{2}})v\|^{2}+\|\op(\Lambda^{\frac{3}{2}})v\|^{2}+\tau|v|_{1,0,\tau}^{2}
\\
 \leq C\left(\|Pv\|^{2}+\tau|\op(b_{1})v|_{1}^{2}+\tau|\op(b_{2})v|^{2}\right).
%\end{split}
\end{multline}

Let $\chi\in\mathcal{C}_{0}^{\infty}(\R^{n})$ such that $\chi\equiv 1$ in the support of $w$. We set $v=\chi\op(\Lambda^{-\frac{1}{2}})w$ and we write
\begin{equation}\label{carl33}
\begin{split}
Pv&=\op(\Lambda^{-\frac{1}{2}})Pw+[P,\op(\Lambda^{-\frac{1}{2}})]w+P[\chi,\op(\Lambda^{-\frac{1}{2}})]w
\\
&=\op(\Lambda^{-\frac{1}{2}})Pw+[P,\op(\Lambda^{-\frac{1}{2}})]w+D_{x_{n}}^{2}[\chi,\op(\Lambda^{-\frac{1}{2}})]w+R[\chi,\op(\Lambda^{-\frac{1}{2}})]w
\\
&\quad +\tau c_{0}(x)D_{x_{n}}[\chi,\op(\Lambda^{-\frac{1}{2}})]w+\tau C_{1}(x)[\chi,\op(\Lambda^{-\frac{1}{2}})]w+\tau^{2}c_{0}'(x)[\chi,\op(\Lambda^{-\frac{1}{2}})]w.
\end{split}
\end{equation}
We have $[\chi,\op(\Lambda^{-\frac{1}{2}})]\in\mathcal{TO}^{-\frac{3}{2}}$, then
\begin{equation}\label{carl34}
\left\|D_{x_{n}}^{2}[\chi,\op(\Lambda^{-\frac{1}{2}})]w\right\|^{2}\leq C\left(\left\|D_{x_{n}}^{2}\op(\Lambda^{-\frac{3}{2}})w\right\|^{2}+\left\|D_{x_{n}}\op(\Lambda^{-\frac{3}{2}})w\right\|^{2}+\left\|\op(\Lambda^{-\frac{3}{2}})w\right\|^{2}\right),
\end{equation}
and
\begin{equation}\label{carl35}
\tau^{2}\left\|c_{0}(x)D_{x_{n}}[\chi,\op(\Lambda^{-\frac{1}{2}})]w\right\|^{2}\leq C\tau^{2}\left(\left\|D_{x_{n}}\op(\Lambda^{-\frac{3}{2}})w\right\|^{2}+\left\|\op(\Lambda^{-\frac{3}{2}})w\right\|^{2}\right)  .
\end{equation}
Since $R[\chi,\op(\Lambda^{-\frac{1}{2}})]\in\mathcal{TO}^{\frac{1}{2}}$, $C_{1}(x)[\chi,\op(\Lambda^{-\frac{1}{2}})]\in\mathcal{TO}^{-\frac{1}{2}}$ and $c_{0}'(x)[\chi,\op(\Lambda^{-\frac{1}{2}})]\in\mathcal{TO}^{-\frac{3}{2}}$, we obtain
\begin{equation}\label{carl36}
\left\|R[\chi,\op(\Lambda^{-\frac{1}{2}})]w\right\|^{2}+\tau^{2}\left\|C_{1}(x)[\chi,\op(\Lambda^{-\frac{1}{2}})]w\right\|^{2}+\tau^{4}\left\|c_{0}'(x)[\chi,\op(\Lambda^{-\frac{1}{2}})]w\right\|^{2}\leq C\left\|\op(\Lambda^{\frac{1}{2}})w\right\|^{2}.
\end{equation}
Since we can write
$$
[P,\op(\Lambda^{-\frac{1}{2}})]=[R,\op(\Lambda^{-\frac{1}{2}})]+\tau[c_{0}(x)D_{x_{n}},\op(\Lambda^{-\frac{1}{2}})]+\tau[C_{1}(x),\op(\Lambda^{-\frac{1}{2}})]+\tau^{2}[c_{0}'(x),\op(\Lambda^{-\frac{1}{2}})],
$$
then by using \eqref{carl12}--\eqref{carl15}, we obtain
\begin{equation}\label{carl37}
\left\|[P,\op(\Lambda^{-\frac{1}{2}})]w\right\|^{2}\leq C\left(\left\|\op(\Lambda^{\frac{1}{2}})w\right\|^{2}+\tau^{-1}\left\|D_{x_{n}}w\right\|^{2}\right).
\end{equation}
Inserting \eqref{carl34}-\eqref{carl37} into \eqref{carl33}, we find
\begin{equation}\label{carl38}
\|Pv\|^{2}\leq C\left(\tau^{-1}\|Pw\|^{2}+\tau^{-1}\|\op(\Lambda)w\|^{2}+\tau^{-1}\|D_{x_{n}}w\|^{2}+\tau^{-1}\|D_{x_{n}}^{2}\op(\Lambda^{-1})w\|^{2}\right).
\end{equation}
We have 
$$
\op(b_{1})v=\op(b_{1})\chi\op(\Lambda^{-\frac{1}{2}})w=\op(\Lambda^{-\frac{1}{2}})\op(b_{1})w+\op(b_{1})[\chi,\op(\Lambda^{-\frac{1}{2}})]w+[\op(b_{1}),\op(\Lambda^{-\frac{1}{2}})]w.
$$
Since $\op(b_{1})\in\mathcal{TO}^{0}$ then $\op(b_{1})[\chi,\op(\Lambda^{-\frac{1}{2}})]\in\mathcal{TO}^{-\frac{3}{2}}$ and $[\op(b_{1}),\op(\Lambda^{-\frac{1}{2}})]\in\mathcal{TO}^{-\frac{3}{2}}$ which gives
\begin{equation}\label{carl39}
\begin{split}
\tau|\op(b_{1})v|_{1}^{2}=\tau|\op(\Lambda)\op(b_{1})v|^{2}&\leq C\left(\tau|\op(\Lambda^{\frac{1}{2}})\op(b_{1})w|^{2}+|\op(\Lambda^{-\frac{1}{2}})w|^{2}\right)
\\
&\leq C\left(\tau|\op(\Lambda^{\frac{1}{2}})\op(b_{1})w|^{2}+\tau^{-2}|\op(\Lambda^{\frac{1}{2}})w|^{2}\right).
\end{split}
\end{equation}
We have
$$
\op(b_{2})v=\op(b_{2})\chi\op(\Lambda^{-\frac{1}{2}})w=\op(\Lambda^{-\frac{1}{2}})\op(b_{2})w+\op(b_{1})[\chi,\op(\Lambda^{-\frac{1}{2}})]w+[\op(b_{2}),\op(\Lambda^{-\frac{1}{2}})]w.
$$
Since $\op(b_{2})\in D_{x_{n}}+\mathcal{TO}^{1}$ then it is clear that $\op(b_{2})[\chi,\op(\Lambda^{-\frac{1}{2}})]\in\mathcal{TO}^{-\frac{3}{2}}D_{x_{n}}+\mathcal{TO}^{-\frac{1}{2}}$ and $[\op(b_{2}),\op(\Lambda^{-\frac{1}{2}})]\in\mathcal{TO}^{-\frac{3}{2}}D_{x_{n}}+\mathcal{TO}^{-\frac{1}{2}}$ hence
\begin{equation}\label{carl40}
\begin{split}
\tau|\op(b_{2})v|^{2}&\leq C\tau\left(|\op(\Lambda^{-\frac{1}{2}})\op(b_{2})w|^{2}+|\op(\Lambda^{-\frac{1}{2}})w|^{2}+|D_{x_{n}}\op(\Lambda^{-\frac{3}{2}})w|^{2}\right)
\\
&\leq C\left(\tau|\op(\Lambda^{-\frac{1}{2}})\op(b_{2})w|^{2}+\tau^{-1}|\op(\Lambda^{\frac{1}{2}})w|^{2}+\tau^{-1}|D_{x_{n}}\op(\Lambda^{-\frac{1}{2}})w|^{2}\right).
\end{split}
\end{equation}
Moreover, we can write
$$
\op(\Lambda)v=\op(\Lambda)\chi\op(\Lambda^{-\frac{1}{2}})w=\op(\Lambda^{\frac{1}{2}})w+\op(\Lambda)[\chi,\op(\Lambda^{-\frac{1}{2}})]w,
$$
since $\op(\Lambda)[\chi,\op(\Lambda^{-\frac{1}{2}})]\in\mathcal{TO}^{-\frac{1}{2}}$  then we get
$$
\tau|\op(\Lambda)v|^{2}\geq\tau|\op(\Lambda^{\frac{1}{2}})w|^{2}-C\tau|\op(\Lambda^{-\frac{1}{2}})w|^{2}\geq\tau|\op(\Lambda^{\frac{1}{2}})w|^{2}-C\tau^{-1}|\op(\Lambda^{\frac{1}{2}})w|^{2},
$$
and for $\tau$ large enough we obtain
\begin{equation}\label{carl41}
\tau|\op(\Lambda^{\frac{1}{2}})w|^{2}\leq C\tau|\op(\Lambda)v|^{2}.
\end{equation}
By using \eqref{carl41} similarly we can prove that for $\tau$ large enough we have
\begin{equation}\label{carl42}
\tau|D_{x_{n}}\op(\Lambda^{-\frac{1}{2}})w|^{2}\leq C\tau|D_{x_{n}}v|^{2}+ C\tau  |v |_1^{2}.
\end{equation}
Recalling that
$$
\tau|v|_{1,0,\tau}^{2}=\tau|v|_{1}^{2}+\tau|D_{n}v|^{2}=\tau|\op(\Lambda)v|^{2}+\tau|D_{n}v|^{2},
$$
and combining \eqref{carl41} and \eqref{carl42}, we obtain
\begin{equation}\label{carl43}
\tau|\op(\Lambda^{\frac{1}{2}})w|^{2}+\tau|D_{x_{n}}\op(\Lambda^{-\frac{1}{2}})w|^{2}\leq C\tau|v|_{1,0,\tau}^{2}.
\end{equation}
Since we have
$$
\op(\Lambda^{\frac{3}{2}})v=\op(\Lambda^{\frac{3}{2}})\chi\op(\Lambda^{-\frac{1}{2}})w=\op(\Lambda)w+\op(\Lambda^{\frac{3}{2}})[\chi,\op(\Lambda^{-\frac{1}{2}})]w
$$
where $\op(\Lambda^{\frac{3}{2}})[\chi,\op(\Lambda^{-\frac{1}{2}})]\in\mathcal{TO}^{0}$ we obtain
\begin{equation}\label{carl44}
\|\op(\Lambda)w\|^{2}-C\|w\|^{2}\leq\|\op(\Lambda^{\frac{3}{2}})v\|^{2}.
\end{equation}
Similarly we can prove also that
\begin{equation}\label{carl45}
\|D_{x_{n}}w\|^{2}-C\left(\|D_{x_{n}}\op(\Lambda^{-1})w\|^{2}+\|\op(\Lambda^{-1})w\|^{2}\right)\leq \|D_{x_{n}}\op(\Lambda^{\frac{1}{2}})v\|^{2},
\end{equation}
and
\begin{multline}\label{carl46}
%\begin{split}
\|D_{x_{n}}^{2}\op(\Lambda^{-1})w\|^{2}-C\big(\|D_{x_{n}}^{2}\op(\Lambda^{-2})w\|^{2}+\|D_{x_{n}}\op(\Lambda^{-2})w\|^{2}
+\|\op(\Lambda^{-2})w\|^{2}\big) 
\\ \leq \|D_{x_{n}}^{2}\op(\Lambda^{-\frac{1}{2}})v\|^{2}.
%\end{split}
\end{multline}
Combining \eqref{carl44}--\eqref{carl46} we find
\begin{multline}\label{carl47}
%\begin{split}
\|D_{x_{n}}^{2}\op(\Lambda^{-1})w\|^{2}+\|D_{x_{n}}w\|^{2}+\|\op(\Lambda)w\|^{2}
\\
\leq\|D_{x_{n}}^{2}\op(\Lambda^{-\frac{1}{2}})v\|^{2}
+\|D_{x_{n}}\op(\Lambda^{\frac{1}{2}})v\|^{2}+\|\op(\Lambda^{\frac{3}{2}})v\|^{2}.
%\end{split}
\end{multline}
Inserting \eqref{carl38}--\eqref{carl40}, \eqref{carl43} and \eqref{carl47} into \eqref{carl32}, we obtain
\begin{align*}
%\begin{split}
&\|D_{x_{n}}^{2}\op(\Lambda^{-1})w\|^{2}+\|D_{x_{n}}w\|^{2}+\|\op(\Lambda)w\|^{2}
+\tau|\op(\Lambda^{\frac{1}{2}})w|^{2}+\tau|D_{x_{n}}\op(\Lambda^{-\frac{1}{2}})w|^{2}
\\
&\leq C\Big(\tau^{-1}\|Pw\|^{2}+\tau^{-1}\|\op(\Lambda)w\|^{2}+\tau^{-1}\|D_{x_{n}}w\|^{2}+\tau^{-1}\|D_{x_{n}}^{2}\op(\Lambda^{-1})w\|^{2}+\tau|\op(\Lambda^{\frac{1}{2}})\op(b_{1})w|^{2}
\\
&\qquad \quad +\tau^{-2}|\op(\Lambda^{\frac{1}{2}})w|^{2}+\tau|\op(\Lambda^{-\frac{1}{2}})\op(b_{2})w|^{2}+\tau^{-1}|\op(\Lambda^{\frac{1}{2}})w|^{2}+\tau^{-1}|D_{x_{n}}\op(\Lambda^{-\frac{1}{2}})w|^{2}\Big).
%\end{split}
\end{align*}
For $\tau$ large enough we yield
 \begin{multline*}
%\begin{split}
\|D_{x_{n}}w\|^{2}+\|\op(\Lambda)w\|^{2}+\tau|\op(\Lambda^{\frac{1}{2}})w|^{2}+\tau|D_{x_{n}}\op(\Lambda^{-\frac{1}{2}})w|^{2}
\\
\leq C\left(\tau^{-1}\|Pw\|^{2}+\tau|\op(\Lambda^{\frac{1}{2}})\op(b_{1})w|^{2}+\tau|\op(\Lambda^{-\frac{1}{2}})\op(b_{2})w|^{2}\right),
%\end{split}
\end{multline*}
which obviously leads to the Carleman estimate. And this end the proof.
\end{proof}

For $u=(u_{1},u_{2})\in H^{1}(\mathcal{U}_{1})\times H^{1}(\mathcal{U}_{2})$ we define the tangential operators $\op (B_{1})$  and $\op(B_{2})$ by
\begin{equation}\label{carl1}
\op(B_{1})u=u_{1|\gamma_{0}}-u_{2|\gamma_{0}}\qquad\text{and}\qquad\op(B_{2})u=\partial_{\nu}u_{1|\gamma_{0}}-\partial_{\nu}u_{2|\gamma_{0}}.
\end{equation} 
We note that $\op (B_{1})$ measure the continuity of the displacement of $u$ through the interface $\gamma_{0}$ where $\op(B_{2})$ describe the difference of the flux through $\gamma_{0}$ of the two sides of the interface.
\begin{cor}
Let $\varphi$ satisfies \eqref{carl2}--\eqref{carl3}. There exist $C>0$ and $\tau_{0}>0$ such that for any $\tau\geq\tau_{0}$ we have the following estimate
\begin{equation}\label{carl50}
\tau^{3}\|\e^{\tau\varphi}u\|^{2}+\tau\|\e^{\tau\varphi}\nabla u\|^{2}+\leq C\left(\|\e^{\tau\varphi}P(x,D)u\|^{2}
+\tau^2 |\e^{\tau\varphi}\op(B_{1})u|_{\frac{1}{2}}^{2}+\tau|\e^{\tau\varphi}\op(B_{2})u|^{2}\right)
\end{equation}
for any $u\in\mathcal{C}_{0}^{\infty}(K)$ where $K\subset\overline{W}_{1}$ is a compact subset.
\end{cor}
\begin{proof}
Let $w=\e^{\tau\varphi}u$ and we recall that $P(x,D,\tau)w=\e^{\tau\varphi}P(x,D)u$, $\op(b_{1})w=\e^{\tau\varphi_{1}}.\op(B_{1})u$ and $\op(b_{2})w=\e^{\tau\varphi_{1}}.\op(B_{2})u$ then using the fact that $\varphi_{1}$ and $\varphi_{2}$ have the same trace on $\gamma_{0}$ and estimate \eqref{carl51} we obtain \eqref{carl50}.
\end{proof}
Now we can state the global Carleman estimate in $\mathcal{U}_{1}$ and $\mathcal{U}_{2}$ (defined in the beginning of this section page \pageref{carl53}) which is given by the following theorem.
\begin{thm}\label{carl5}
Assume that $\varphi$ satisfies
\begin{eqnarray}
|\nabla\varphi_{k}(x)|> 0,\;\forall\,x\in\overline{\mathcal{U}}_{k},\quad k=1,2,\label{carl48}
\\
\partial_{\nu}\varphi_{|\gamma}(x)< 0,\label{carl52}%\quad\text{ and }\quad\partial_{\nu}\varphi_{1|\gamma_{1}}(x)\neq 0,
\\
\partial_{\nu}\varphi_{k|\gamma_{0}}(x)>0,\quad k=1,2,
\\
\left(\partial_{\nu}\varphi_{1|\gamma_{0}}(x)\right)^{2}-\left(\partial_{\nu}\varphi_{2|\gamma_{0}}(x)\right)^{2}>1,
\end{eqnarray}
and the sub-ellipticity condition
\begin{equation}\label{carl49}
\exists\,c>0,\;\forall\,(x,\xi)\in\overline{\mathcal{U}}_{k}\times\R^{n},\;p_{k}(x,\xi)=0\,\Longrightarrow\,\left\{\mathrm{Re}(p_{k}),\mathrm{Im}(p_{k})\right\}(x,\xi,\tau)\geq c\langle\xi,\tau\rangle^{3}.
\end{equation}
Then there exist $C>0$ and $\tau_{0}>0$ such that we have the following estimate
%\begin{equation}\label{carl4}
%\begin{split}
%\tau^{3}\|\e^{\tau\varphi}u\|_{L^{2}(\mathcal{U})}^{2}+\tau\|\e^{\tau\varphi}\nabla u\|_{L^{2}(\mathcal{U})}^{2}\leq C\Big(\|\e^{\tau\varphi}Pu\|_{L^{2}(\mathcal{U})}^{2}+\tau^{2}\|\e^{\tau\varphi}\op(B_{1})u\|_{H^{\frac{1}{2}}(\gamma_{0})}^{2}
%\\
%+\tau\|\e^{\tau\varphi}\op(B_{2})u\|_{L^{2}(\gamma_{0})}^{2}+\tau^{3}\|\e^{\tau\varphi_{1}}u_{1}\|_{L^{2}(\gamma_{1})}^{2}+\tau\|\e^{\tau\varphi_{1}}\nabla u_{1}\|_{L^{2}(\gamma_{1})}^{2}\Big)
%\end{split}
%\end{equation}
\begin{align}\label{carl4}
%\begin{split}
&\tau^{3}\|\e^{\tau\varphi}u\|_{L^{2}(\mathcal{U})}^{2}+\tau\|\e^{\tau\varphi}\nabla u\|_{L^{2}(\mathcal{U})}^{2}
\\
&\qquad  \leq C\Big(\|\e^{\tau\varphi}Pu\|_{L^{2}(\mathcal{U})}^{2}+\tau^{2}\|\e^{\tau\varphi}\op(B_{1})u\|_{H^{\frac{1}{2}}(\gamma_{0})}^{2}
+\tau\|\e^{\tau\varphi}\op(B_{2})u\|_{L^{2}(\gamma_{0})}^{2}\Big)   \notag
%\\
%&\qquad \qquad \quad 
%+\tau^{3}\|\e^{\tau\varphi_{1}}u_{1}\|_{L^{2}(\gamma_{1})}^{2}
%+\tau\|\e^{\tau\varphi_{1}}\nabla u_{1}\|_{L^{2}(\gamma_{1})}^{2}\Big),    \notag
%\end{split}
\end{align}
for all $\tau\geq\tau_{0}$ and $u=(u_{1},u_{2})\in H^{2}(\mathcal{U}_{1})\times H^{2}(\mathcal{U}_{2})$ such that $u_{2|\gamma}=0$.
\end{thm}

Actually a weight functions with assumptions \eqref{carl48}-\eqref{carl49} can not exist. So, since the proof of the theorem is local then we can do without the conditions \eqref{carl48} and \eqref{carl52} in some region where the entries is supposed to be vanishing around the critical points of the weight functions and where the damping is active. Next, the missing information will be recuperated with a new entries which vanishing far away when the first do (See next section).
%\textbf{Voir comment r\'ediger cela car on a besoin de mettre tous les points critiques dans un ouvert $V$ relativement compact dans $\{a>0\}$. Il faut supposer $u_{|V}=0$ et non $u_{2|\gamma_{2}}=0$.}
%%%%%%%%%%%%%%%%%%%%%%%%%%%%%%%%%%%%%%%%%%%%%%%%%%%%%%%%%%%%%%%%%%%%%%%%%%%%%%%%%%%%%%%%%%%%%%%%%%%%%%%%%%%%%%%%%%%%%%%%%%%%%%%%%%%%%%%%%%%%%%%%%%%%%%%%%%%%%%%%%%%%%%%%%%%%SECTION%%%%%%%%%%%%%%%%%%%%%%%%%%%%%%%%%%%%%%%%%%%%%%%%%%%%%%%%%%%%%%%%%%%%%%%%%%%%%%%%%%%%%%%%%%%%%%%%%%%%%%%%%%%%%%%%%%%%%%%%%%%%%%%%%%%%%%%%%%%%%%%%%%%%%%%%%%%%%%%%%%%%%%%%%%%%%%%%%%%%%%%%
\section{Stabilization result}\label{stab}
In this section, we will prove the logarithmic stability of the system \eqref{wave1}. To this end, we establish a particular resolvent estimate precisely we will show that for some constant $C>0$ we have
\begin{equation}\label{Swave24}
\|(\mathcal{A}-i\mu\,I)^{-1}\|_{\mathcal{L}(\mathcal{H})}\leq C\e^{C|\mu|},\qquad \forall\,|\mu|\gg 1,
\end{equation}
and then by Burq's result \cite{burq} and  the remark of Duyckaerts \cite[section 7]{Duyckaerts} (see also \cite{batty,borichevtomilov}) we obtain the expected decay rate of the energy.

Let $\mu$ be a real number such that $|\mu|$ is large, and assume that
\begin{equation}\label{Swave1}
(\mathcal{A}-i\mu\,I)(u,v)^{t}=(f,g)^{t},\quad (u,v)\in\mathcal{D}(\mathcal{A}),\quad (f,g)\in\mathcal{H}.
\end{equation}
which can be written as follow
\begin{equation*}
\left\{\begin{array}{ll}
v-i\mu u=f&\text{in }\Omega
\\
\Delta u+\mathrm{div}(a(x)\nabla v)-i\mu v=g&\text{in }\Omega,
\end{array}\right.
\end{equation*}
or equivalently,
\begin{equation}\label{Swave2}
\left\{\begin{array}{ll}
v=f+i\mu u&\text{in }\Omega
\\
\Delta u+i\mu\mathrm{div}(a(x)\nabla u)+\mu^{2}u=g+i\mu f-\mathrm{div}(a(x)\nabla f)&\text{in }\Omega.
\end{array}\right.
\end{equation}
Multiplying the second line of \eqref{Swave2} by $\overline{u}$ and integrating over $\Omega$ then by Green's formula we obtain
\begin{equation}\label{Swave3}
\int_{\Omega}(g-i\mu f)\overline{u}\,\ud x+d\int_{\omega}\nabla f.\nabla\overline{u}\,\ud x=\mu^{2}\int_{\Omega}|u|^{2}\,\ud x-\int_{\Omega}|\nabla u|^{2}\,\ud x-id\mu\int_{\omega}|\nabla u|^{2}\,\ud x.
\end{equation}
Taking the imaginary part of \eqref{Swave3} and using the Cauchy-Schwarz inequality and Poincar\'e inequality we find
\begin{equation}\label{Swave4}
d|\mu|\int_{\omega}|\nabla u|^{2}\,\dd x\leq C\left(\mu^{2}\int_{\Omega}|\nabla f|^{2}\,\ud x+\int_{\Omega}|g|^{2}\,\ud x\right).
\end{equation}
By setting $u=u_{1}\,\mathbb{1}_{\omega}+ u_{2}\, \mathbb{1}_{\Omega\setminus\bar{\omega}}$, $v
=v_{1}\,\mathbb{1}_{\omega}+v_{2}\,\mathbb{1}_{\Omega\setminus\bar{\omega}}$, 
$f=f_{1}\,\mathbb{1}_{\omega}+f_{2}\, \mathbb{1}_{\Omega \setminus \bar{\omega}}$ and 
$g=g_{1}\,\mathbb{1}_{\omega}+g_{2}\,\mathbb{1}_{\Omega\setminus\bar{\omega}}$ system 
\eqref{Swave2} is transformed to the following transmission equation   
\begin{equation}\label{Swave5}
\left\{\begin{array}{ll}
v_{1}=i\mu u_{1}+f_{1}&\text{in }\omega
\\
v_{2}=i\mu u_{2}+f_{2}&\text{in }\Omega\backslash\overline{\omega}
\\
\Delta((1+id\mu) u_{1}+df_{1})+\mu^{2} u_{1}=g_{1}+i\mu f_{1}&\text{in }\omega
\\
\Delta u_{2}+\mu^{2}u_{2}=g_{2}+i\mu f_{2}&\text{in }\Omega\backslash\overline{\omega},
\end{array}\right.
\end{equation}
with the transmission conditions
\begin{equation}\label{Swave6}
\left\{\begin{array}{ll}
u_{1}=u_{2}&\text{on }\mathcal{I}
\\
\partial_{\nu}((1+id\mu)u_{1}+df_{1})=\partial_{\nu}u_{2}&\text{on }\mathcal{I},
\end{array}\right.
\end{equation}
and the boundary condition
\begin{equation}\label{Swave7}
u_{2}=0\quad\text{on }\Gamma,
\end{equation}
where $\nu(x)$ denote the outer unit normal to $\Omega\setminus\omega$ on $\Gamma$ and on $\mathcal{I}$ (see Figure \ref{fig1}). 

To prove Theorem \ref{LogStab} we need the following technical lemma
\begin{lem} 
\label{lem: Swave8}
Let $\mathcal{O}$ be a bounded open set of $\R^{n}$. Then there exist $C>0$ and $\mu_{0}>0$, such that for any $w$ and $F$ satisfying
$$
\Delta w+\frac{\mu^{2}}{1+id\mu}w=F\quad\text{in }\mathcal{O}
$$
and for all $|\mu|>\mu_{0}$ we have the following estimate
\begin{equation}\label{Swave8}
\|w\|_{H^{1}}^{2}\leq C\left(\|\nabla w\|_{L^{2}(\mathcal{O})}^{2}+\|F\|_{L^{2}(\mathcal{O})}^{2}\right).
\end{equation} 
\end{lem}
\begin{proof}
We need to distinguish two cases

\underline{Inside $\mathcal{O}$}: Let $\chi\in\mathcal{C}_{0}^{\infty}(\mathcal{O})$, we have by integration by parts
$$
\int_{\mathcal{O}}\left(\Delta w+\frac{\mu^{2}}{1+id\mu}w\right).\chi^{2}\overline{w}\,\ud x=\frac{\mu^{2}}{1+id\mu}\|\chi w\|_{L^{2}(\mathcal{O})}^{2}-\int_{\mathcal{O}}|\chi \nabla w|^{2}\,\ud x-2\int_{\mathcal{O}}\nabla\chi.\nabla w\chi\overline{w}\,\ud x. 
$$
Then we obtain
$$
\frac{\mu^{2}}{\sqrt{1+d^{2}\mu^{2}}}\|\chi w\|_{L^{2}(\mathcal{O})}^{2}\leq C\left(\|F\|_{L^{2}(\mathcal{O})}.\|\chi^{2}w\|_{L^{2}(\mathcal{O})}+\|\nabla w\|_{L^{2}(\mathcal{O})}^{2}+\|\nabla w\|_{L^{2}(\mathcal{O})}.\|\chi w\|_{L^{2}(\mathcal{O})}\right).
$$
Using Cauchy-Schwarz inequality and for $|\mu|$ large enough, one gets
\begin{equation}\label{Swave9}
\|\chi w\|_{L^{2}(\mathcal{O})}^{2}\leq C\left(\|\nabla w\|_{L^{2}(\mathcal{O})}^{2}+\|F\|_{L^{2}(\mathcal{O})}^{2}\right).
\end{equation}
hence the result inside $\mathcal{O}$.

\underline{In the neighborhood of the boundary}: Let $x=(x',x_{n})\in\R^{n-1}\times\R$. then
$$
\partial\mathcal{O}=\{x\in\R^{n},\; x_{n}=0\}.
$$
Let $\varepsilon>0$ such that $0<x_{n}<\varepsilon$. Then we have
$$
w(x,\varepsilon)-w(x',x_{n})=\int_{x_{n}}^{\varepsilon}\partial_{x_{n}}w(x',t)\,\ud t.
$$
It follows
$$
|w(x',x_{n})|^{2}\leq 2|w(x',\varepsilon)|^{2}+2\left(\int_{x_{n}}^{\varepsilon}|\partial_{x_{n}}w(x',t)|\,\ud t\right)^{2}.
$$
Using the Cauchy-Schwarz inequality, we obtain
$$
|w(x',x_{n})|^{2}\leq 2|w(x',\varepsilon)|^{2}+2\varepsilon \int_{x_{n}}^{\varepsilon}|\partial_{x_{n}}w(x',t)|^{2}\,\ud t.
$$
Integrating with respect to $x'$, we obtain
\begin{equation}\label{Swave10}
\int_{|x'|<\varepsilon}|w(x',x_{n})|^{2}\,\ud x'\leq 2\int_{|x'|<\varepsilon}|w(x',\varepsilon)|^{2}\,\ud x'+2\varepsilon\int_{|x'|<\varepsilon}\int_{|x_{n}|<\varepsilon}|\partial_{x_{n}}w(x',t)|^{2}\,\ud t\,\ud x'.
\end{equation}
Using the trace theorem, we have
\begin{equation}\label{Swave11}
\int_{|x'|<\varepsilon}|w(x',\varepsilon)|^{2}\,\ud x'\leq C\int_{|x'|<2\varepsilon,|x_{n}-\varepsilon|<\frac{\varepsilon}{2}}\left(|w(x)|^{2}+|\nabla w(x)|^{2}\right)\,\ud x.
\end{equation}
We introduce the following cut-off functions
$$
\chi_{1}(x)=\left\{\begin{array}{lll}
1&\text{if}&\ds 0<x_{n}<\frac{\varepsilon}{2}
\\
0&\text{if}&\ds x_{n}>\varepsilon,
\end{array}\right.
$$
and
$$
\chi_{2}(x)=\left\{\begin{array}{lll}
1&\text{if}&\ds\frac{\varepsilon}{2}<x_{n}<\frac{3\varepsilon}{2}
\\
0&\text{if}&\ds x_{n}<\frac{\varepsilon}{4},\quad x_{n}>2\varepsilon.
\end{array}\right.
$$
Combining \eqref{Swave10} and \eqref{Swave11}, we obtain for $\varepsilon$ small enough
\begin{equation}\label{Swave12}
\|\chi_{1} w\|^{2}\leq C\left(\|\chi_{2}w\|^{2}+\|\nabla w\|^{2}\right).
\end{equation}
From \eqref{Swave9}, we have
\begin{equation}\label{Swave13}
\|\chi_{2} w\|^{2}\leq C\left(\|\nabla w\|^{2}+\|F\|^{2}\right).
\end{equation}
Inserting \eqref{Swave13} into \eqref{Swave12} we find
\begin{equation}\label{Swave14}
\|\chi_{1} w\|^{2}\leq C\left(\|\nabla w\|^{2}+\|F\|^{2}\right).
\end{equation}
hence the result in the neighborhood of the boundary.

Following to \eqref{Swave9}, we can write
\begin{equation}\label{Swave15}
\|(1-\chi_{1})w\|^{2}\leq C\left(\|\nabla w\|^{2}+\|F\|^{2}\right).
\end{equation}
Adding \eqref{Swave14} and \eqref{Swave15} we obtain \eqref{Swave8}.
\end{proof}
Now we can prove Theorem \ref{LogStab}. We set $w_{1}=(1+id\mu)u_{1}+df_{1}$ and $w_{2}=u_{2}$, then the system \eqref{Swave5}-\eqref{Swave7} can be recast as follow
\begin{equation}\label{Swave16}
\left\{\begin{array}{ll}
\ds\Delta w_{1}+\frac{\mu^{2}}{1+id\mu}w_{1}=\Phi_{1}&\text{in }\omega
\\
\ds\Delta w_{2}+\mu^{2}w_{2}=\Phi_{2}&\text{in }\Omega\setminus\omega,
\end{array}\right.
\end{equation}
with the transmission conditions
\begin{equation}\label{Swave17}
\left\{\begin{array}{ll}
w_{1}=w_{2}+\phi&\text{on }\mathcal{I}
\\
\partial_{\nu}w_{1}=\partial_{\nu}w_{2}&\text{on }\mathcal{I},
\end{array}\right.
\end{equation}
and the boundary condition
\begin{equation}\label{Swave18}
\begin{array}{ll}
w_{2}=0&\text{on }\Gamma,
\end{array}
\end{equation}
where we have denoted by $\ds\Phi_{1}=g_{1}+\frac{i\mu}{1+id\mu}f_{1}$, $\ds\Phi_{2}=g_{2}+i\mu f_{2}$ and $\ds\phi=df_{1}+id\mu u_{1}$.

%\textbf{Je trouve  $\ds\Phi_{1}=g_{1}+\frac{i\mu-d  ( 1-\mu^{2})}{1+i\mu}f_{1}$. A v\'erifier...}
We denoted by $B_{r}$ a ball of radius $r>0$ in $\omega$ and $B_{r}^{c}$ its complementary such that $B_{4r}\subset\omega$. Let's introduce the cut-off function $\chi\in\mathcal{C}^{\infty}(\omega)$ by
$$
\chi(x)=\left\{\begin{array}{ll}
1&\text{in } B_{3r}^{c}
\\
0&\text{in } B_{2r}.
\end{array}\right.
$$
Next, we denote by $\widetilde{w}_{1}=\chi w_{1}$ then from the first line of \eqref{Swave16}, one sees that
\begin{equation}\label{Swave19}
\begin{array}{ll}
\ds\Delta\widetilde{w}_{1}+\frac{\mu^{2}}{1+id\mu}\widetilde{w}_{1}=\widetilde{\Phi}_{1}&\text{in }\omega,
\end{array}
\end{equation}
where $\widetilde{\Phi}_{1}=\chi\Phi_{1}-[\Delta,\chi]w_{1}$.  We denote by $\Omega_{1}=\omega\setminus\overline{B}_{r}$ and $\Omega_{2}=\Omega\setminus\overline{\omega}$.

According to \cite{burq}, \cite{hassine2} or \cite{hassine3} we can find four weight functions $\varphi_{1,1}$, $\varphi_{1,2}$, $\varphi_{2,1}$ and $\varphi_{2,2}$, a finite number of points $x_{j,k}^{i}$ where $\overline{B(x_{j,k}^{i},2\varepsilon)}\subset\Omega_{j}$ for all $j,k=1,2$ and $i=1,\ldots,N_{i,k}$ such that by denoting $U_{j,k}=\ds\Omega_{k}\bigcap\left(\bigcup_{i=1}^{N_{j,k}}\overline{B(x_{j,k}^{i},\epsilon)}\right)^{c}$, the weight function $\varphi_{k}=\mathrm{diag}(\varphi_{1,k},\varphi_{2,k})$ verifying the assumption \eqref{carl48}-\eqref{carl49} in $U_{1,k}\cup U_{2,k}$ with $\gamma_{1}=\partial B_{r}$, $\gamma_{2}=\Gamma$ and $\gamma=\mathcal{I}$. Moreover, $\varphi_{j,k}<\varphi_{j,k+1}$ in $\ds\bigcup_{i=1}^{N_{j,k}}B(x_{j,k}^{i},2\epsilon)$ for all $j,k=1,2$ where we denoted by $\varphi_{j,3}=\varphi_{j,1}$.

Let $\chi_{j,k}$ (for $j,k=1,2$) four cut-off functions equal to $1$ in $\ds\left(\bigcup_{i=1}^{N_{j,k}}B(x_{j,k}^{i},2\epsilon)\right)^{c}$ and supported in $\ds\left(\bigcup_{i=1}^{N_{j,k}}B(x_{j,k}^{i},\epsilon)\right)^{c}$ (in order to eliminate the critical points of the weight functions $\varphi_{j,k}$). We set $w_{1,1}=\chi_{1,1}\widetilde{w}_{1}$, $w_{1,2}=\chi_{1,2}\widetilde{w}_{1}$, $w_{2,1}=\chi_{2,1}w_{2}$ and $w_{2,2}=\chi_{2,2}w_{2}$. Then from system \eqref{Swave17} and equations \eqref{Swave7} and \eqref{Swave19}, then for $k=1,2$ we obtain
\begin{equation}\label{Swave20}
\left\{\begin{array}{ll}
\ds\Delta w_{1,k}+\frac{\mu^{2}}{1+id\mu}w_{1,k}=\Psi_{1,k}&\text{in }\omega
\\
\Delta w_{2,k}+\mu^{2}w_{2,k}=\Psi_{2,k}&\text{in }\Omega\setminus\omega
\\
w_{1,k}=w_{2,k}+\phi&\text{on }\mathcal{I}
\\
\partial_{\nu}w_{1,k}=  \partial_{\nu} w_{2,k}&\text{on }\mathcal{I}
\\
w_{2,k}=0&\text{on }\Gamma,
\end{array}\right.
\end{equation}
where
\begin{equation}\label{Swave21}
\left\{\begin{array}{l}
\Psi_{1,k}=\chi_{1,k}\widetilde{\Phi}_{1}-[\Delta,\chi_{1,k}]\widetilde{w}_{1}
\\
\Psi_{2,k}=\chi_{2,k}\Phi_{2}-[\Delta,\chi_{2,k}]w_{2}.
\end{array}\right.
\end{equation}
Applying now Carleman estimate \eqref{carl4} to the system \eqref{Swave20} with $\tau=|\mu|$ then for $k=1,2$ we have
\begin{multline*}
%\begin{split}
\tau^{3 }\sum_{j=1,2}\|\e^{\tau\varphi_{j,k}}w_{j,k}\|_{L^{2}(U_{j,k})}^{2}
%+\tau^{3}\|\e^{\tau\varphi_{2,k}}w_{2,k}\|_{L^{2}(U_{2,k})}^{2}
+\tau\sum_{j=1,2}\|\e^{\tau\varphi_{j,k}}\nabla w_{j,k}\|_{L^{2}(U_{j,k})}^{2}
%\tau^{3}\|\e^{\tau\varphi_{1,k}}w_{1,k}\|_{L^{2}(U_{1,k})}^{2}+\tau^{3}\|\e^{\tau\varphi_{2,k}}w_{2,k}\|_{L^{2}(U_{2,k})}^{2}+\tau\|\e^{\tau\varphi_{1,k}}\nabla w_{1,k}\|_{L^{2}(U_{1,k})}^{2}
%+\tau\|\e^{\tau\varphi_{2,k}}\nabla w_{2,k}\|_{L^{2}(U_{2,k})}^{2}
\\
\leq C\Big(\|\e^{\tau\varphi_{1,k}}\Psi_{1,k}\|_{L^{2}(U_{1,k})}^{2}
+\|\e^{\tau\varphi_{2,k}}\Psi_{2,k}\|_{L^{2}(U_{2,k})}^{2}
+\tau^{2}\|\e^{\tau\varphi_{1,k}}\phi\|_{H^{\frac{1}{2}}(\mathcal{I})}^{2}\Big).
%\end{split}
\end{multline*}
We recall the expression of $\Psi_{1,k}$ and $\Psi_{2,k}$ in \eqref{Swave21}, then we can write
\begin{multline*}
%\begin{split}
\tau^{3 }\sum_{j=1,2}\|\e^{\tau\varphi_{j,k}}w_{j,k}\|_{L^{2}(U_{j,k})}^{2}
%+\tau^{3}\|\e^{\tau\varphi_{2,k}}w_{2,k}\|_{L^{2}(U_{2,k})}^{2}
+\tau\sum_{j=1,2}\|\e^{\tau\varphi_{j,k}}\nabla w_{j,k}\|_{L^{2}(U_{j,k})}^{2}
%+\tau\|\e^{\tau\varphi_{2,k}}\nabla w_{2,k}\|_{L^{2}(U_{2,k})}^{2}
\\
\leq C\Big(\|\e^{\tau\varphi_{1,k}}\Phi_{1}\|_{L^{2}(U_{1,k})}^{2}+\|\e^{\tau\varphi_{2,k}}\Phi_{2}\|_{L^{2}(U_{2,k})}^{2}
%\\
+\|\e^{\tau\varphi_{1,k}}[\Delta,\chi_{1,k}]\widetilde{w}_{1}\|_{L^{2}(U_{1,k})}^{2}  \\
+\|\e^{\tau\varphi_{1,k}}[\Delta,\chi]w_{1}\|_{L^{2}(U_{1,k})}^{2}
+\|\e^{\tau\varphi_{2,k}}[\Delta,\chi_{2,k}]w_{2,k}\|_{L^{2}(U_{2,k})}^{2}+\tau^{2}\|\e^{\tau\varphi_{1,k}}\phi\|_{H^{\frac{1}{2}}(\mathcal{I})}^{2}\Big).
%\end{split}
\end{multline*}
Adding the two last estimates and using the property of the weight functions $\varphi_{j,1}<\varphi_{1,2}$ in 
$\ds\bigcup_{i=1}^{N_{j,1}}B(x_{j,1}^{i},2\epsilon)$ and $\varphi_{j,2}<\varphi_{j,1}$ in $\ds\bigcup_{i=1}^{N_{j,2}}B(x_{j,2}^{i},2\epsilon)$ 
for all $j=1,2$, then we can absorb first order the terms $[\Delta,\chi_{1,k}]\widetilde{w}_{1}$ and $[\Delta,\chi_{2,k}]w_{2}$ at the right 
hand side into the left hand side for $\tau>0$  sufficiently large, mainly we obtain
\begin{multline*}
%\begin{split}
\tau\int_{\Omega_{1}}\left(\e^{2\tau\varphi_{1,1}}+\e^{2\tau\varphi_{1,2}}\right)|\nabla w_{1}|^{2}\,\ud x+\tau\int_{\Omega_{2}}\left(\e^{2\tau\varphi_{2,1}}+\e^{2\tau\varphi_{2,2}}\right)|\nabla w_{2}|^{2}\,\ud x
\\
\leq C\bigg(\int_{\Omega_{1}}\left(\e^{2\tau\varphi_{1,1}}+\e^{2\tau\varphi_{1,2}}\right)|\Phi_{1}|^{2}\,\ud x+\tau\int_{\Omega_{2}}\left(\e^{2\tau\varphi_{2,1}}+\e^{2\tau\varphi_{2,2}}\right)|\Phi_{2}|^{2}\,\ud x
\\
 +\int_{\Omega_{1}}\left(\e^{2\tau\varphi_{1,1}}+\e^{2\tau\varphi_{1,2}}\right)|[\Delta,\chi]w_{1}|^{2}\,\ud x+\tau^{2}\left(\|\e^{\tau\varphi_{1,1}}\phi\|_{H^{\frac{1}{2}}(\mathcal{I})}^{2}+\|\e^{\tau\varphi_{1,2}}\phi\|_{H^{\frac{1}{2}}(\mathcal{I})}^{2}\right)\bigg).
%\end{split}
\end{multline*}
Since we can write $\ds\phi=\frac{id\mu}{1+id\mu}w_{1}+\frac{d}{1+id\mu}f_{1}$ then using the trace theorem, Green's formula and the fact that the operator $[\Delta,\chi]$ is of the first order with support in $\omega$ we find
%\textbf{Je trouve $\ds\phi=\frac{id\mu}{1+id\mu}w_{1}+\frac{1-id^2-id^2\mu}{1+id\mu}f_{1}$. A v\'erifier...}
\begin{multline}\label{Swave22}
%\begin{split}
\tau\int_{\omega}\left(\e^{2\tau\varphi_{1,1}}+\e^{2\tau\varphi_{1,2}}\right)|\nabla w_{1}|^{2}\,\ud x+\tau\int_{\Omega\setminus\omega}\left(\e^{2\tau\varphi_{2,1}}+\e^{2\tau\varphi_{2,2}}\right)|\nabla w_{2}|^{2}\,\ud x
\\
\leq C\bigg(\int_{\omega}\left(\e^{2\tau\varphi_{1,1}}+\e^{2\tau\varphi_{1,2}}\right)|\Phi_{1}|^{2}\,\ud x+\tau\int_{\Omega\setminus\omega}\left(\e^{2\tau\varphi_{2,1}}+\e^{2\tau\varphi_{2,2}}\right)|\Phi_{2}|^{2}\,\ud x
\\
+\tau^{4}\int_{\omega}\left(\e^{2\tau\varphi_{1,1}}+\e^{2\tau\varphi_{1,2}}\right)|w_{1}|^{2}\,\ud x+\tau^{2}\int_{\omega}\left(\e^{2\tau\varphi_{1,1}}+\e^{2\tau\varphi_{1,2}}\right)|\nabla w_{1}|^{2}\,\ud x
\\
+\tau^{4}\int_{\omega}\left(\e^{2\tau\varphi_{1,1}}+\e^{2\tau\varphi_{1,2}}\right)|f_{1}|^{2}\,\ud x+\tau^{2}\int_{\omega}\left(\e^{2\tau\varphi_{1,1}}+\e^{2\tau\varphi_{1,2}}\right)|\nabla f_{1}|^{2}\,\ud x\bigg).
%\end{split}
\end{multline}
Using the expression of $\Phi_{1}$ and $\Phi_{2}$,  taking the maximum of $\varphi_{1,1}$, $\varphi_{1,2}$, $\varphi_{2,1}$ and $\varphi_{2,2}$ in the right hand side of \eqref{Swave22} and their minimum in the left hand side and Lemma \ref{lem: Swave8} we follow  
\begin{equation*}
\begin{split}
\|\nabla w_{1}\|_{L^{2}(\omega)}^{2}+\|\nabla w_{2}\|_{L^{2}(\Omega\setminus\omega)}^{2}\leq C\e^{C\tau}\Big(\|f_{1}\|_{L^{2}(\omega)}^{2}+\|\nabla f_{1}\|_{L^{2}(\omega)}^{2}+\|f_{2}\|_{L^{2}(\Omega\setminus\omega)}^{2}
\\
+\|g_{1}\|_{L^{2}(\omega)}^{2}+\|g_{2}\|_{L^{2}(\Omega\setminus\omega)}^{2}+\|\nabla w_{1}\|_{L^{2}(\omega)}^{2}\Big).
\end{split}
\end{equation*}
We evoke $u_{1}$ and $u_{2}$ through the expression of $w_{1}$ and $w_{2}$, one gets
\begin{equation*}
\begin{split}
\|\nabla u_{1}\|_{L^{2}(\omega)}^{2}+\|\nabla u_{2}\|_{L^{2}(\Omega\setminus\omega)}^{2}\leq C\e^{C\tau}\Big(\|f_{1}\|_{L^{2}(\omega)}^{2}+\|\nabla f_{1}\|_{L^{2}(\omega)}^{2}+\|f_{2}\|_{L^{2}(\Omega\setminus\omega)}^{2}
\\
+\|g_{1}\|_{L^{2}(\omega)}^{2}+\|g_{2}\|_{L^{2}(\Omega\setminus\omega)}^{2}+\|\nabla u_{1}\|_{L^{2}(\omega)}^{2}\Big).
\end{split}
\end{equation*}
Using the Poincar\'e inequality then we have
\begin{equation}\label{Swave23}
\|\nabla u\|_{L^{2}(\Omega)}^{2}\leq C\e^{c|\mu|}\left(\|\nabla f\|_{L^{2}(\Omega)}^{2}+\|g\|_{L^{2}(\Omega)}^{2}+\|\nabla u\|_{L^{2}(\omega)}^{2}\right).
\end{equation}
The combination of the two estimates \eqref{Swave4} and \eqref{Swave23} leads to
\begin{equation}\label{Swave25}
\|\nabla u\|_{L^{2}(\Omega)}^{2}\leq C\e^{c|\mu|}\left(\|\nabla f\|_{L^{2}(\Omega)}^{2}+\|g\|_{L^{2}(\Omega)}^{2}\right).
\end{equation}
We can obtain the same estimate as \eqref{Swave25} with the $v$ variable with the $L^{2}$ norm instead of $u$ by using again the Poincar\'e inequality and recalling the expression of $v$ in the first line of \eqref{Swave2} namely, we have
\begin{equation}\label{Swave26}
\|v\|_{L^{2}(\Omega)}^{2}\leq C\e^{c|\mu|}\left(\|\nabla f\|_{L^{2}(\Omega)}^{2}+\|g\|_{L^{2}(\Omega)}^{2}\right).
\end{equation}
So that, the estimate \eqref{Swave24} is obtained by the combination of the two estimates \eqref{Swave25} and \eqref{Swave26}.
%%%%%%%%%%%%%%%%%%%%%%%%%%%%%%%%%%%%%%%%%%%%%%%%%%%%%%%%%%%%%%%%%%%%%%%%%%%%%%%%%%%%%%%%%%%%%%%%%%%%%%%%%%%%%%%%%%%%%%%%%%%%%%%%%%%%%%%%%%%%%%%%%%%%%%%%%%%%%%%%%%%%%%%%Bibliography%%%%%%%%%%%%%%%%%%%%%%%%%%%%%%%%%%%%%%%%%%%%%%%%%%%%%%%%%%%%%%%%%%%%%%%%%%%%%%%%%%%%%%%%%%%%%%%%%%%%%%%%%%%%%%%%%%%%%%%%%%%%%%%%%%%%%%%%%%%%%%%%%%%%%%%%%%%%%%%%%%%%%%%%%%%%%%%%%

\end{document}